\documentclass[11pt]{amsart}

\usepackage[a4paper,margin=1.05in]{geometry}
\usepackage{amsmath,amssymb,amsthm,mathtools}
\usepackage{microtype}
\usepackage[colorlinks=true,linkcolor=blue,citecolor=blue,urlcolor=blue]{hyperref}

\numberwithin{equation}{section}

\newtheorem{theorem}{Theorem}[section]
\newtheorem{proposition}[theorem]{Proposition}
\newtheorem{corollary}[theorem]{Corollary}

\newtheorem{remark}[theorem]{Remark}
\newtheorem{lemma}[theorem]{Lemma}
\newtheorem{question}[theorem]{Question}
\newcommand{\R}{\mathbb R}
\newcommand{\dimH}{\dim_{\mathcal H}}
\newcommand{\eps}{\varepsilon}
\newcommand{\supp}{\operatorname{supp}}
\newcommand{\Vol}{\mathcal V}

\title{Pinned Nonempty Interior and Volumes of Simplices}

\author[E. A. Palsson]{Eyvindur Ari Palsson}\email{palsson@vt.edu}\address{Department of Mathematics, Virginia Tech, Virginia, VA 24061, USA}
\author[G. Psaromiligkos]{Georgios Psaromiligkos}\email{gpsaromiligkos@uth.gr}\address{Department of Mathematics, University of Thessaly, Lamia, Greece}

\begin{document}

\begin{abstract}
We study pinned nonempty-interior problems for scalar two-point configurations and for volumes of simplices. For $E\subset\mathbb{R}^d$, $d\geq 2$, compact and a smooth scalar configuration map $\Phi(x,y)$, whose corresponding localized generalized Radon transforms are nondegenerate Fourier integral operators of smoothing order $(d-1)/2$, we first note how a calculation due to Greenleaf, Iosevich and Taylor can be used to obtain positive Lebesgue measure of $\Delta_\Phi^y(E)=\{\Phi(x,y):x\in E\}$ for almost every pin $y$ when $\dimH(E)>(d+1)/2$. Our first main result is to prove that the corresponding one-frequency-loss estimate for differentiation in the level parameter yields a continuous pinned density, and hence nonempty interior, for almost every pin when $d\geq3$ and $\dimH(E)>(d+2)/2$. Concrete applications include generalized norm distances, regular variable-coefficient and Riemannian distances, and dot products or nondegenerate bilinear forms on regular patches.

Our principal geometric application concerns volumes of simplices. We prove a cylinder-averaging estimate for triangle areas in $\mathbb{R}^d$ and obtain positive measure for doubly pinned area sets at a dimensional threshold $(d+1)/2$ and nonempty interior at $(d+2)/2$. A projection theorem then reduces higher simplex-volume problems to triangle areas. In particular, for $3\leq k \leq d$, if $\dimH(E)>(d+k-1)/2$, then for every prescribed base point $x_0$ and every prescribed second vertex $y\in E\setminus\{x_0\}$, the set of $k$-dimensional volumes generated by $x_0,y$ and $k-1$ further points of $E$ has nonempty interior. Thus the result is doubly strongly pinned in its first two vertices.
\end{abstract}

\maketitle

\section{Introduction}\label{sec:introduction}

\subsection{Distances, pinned distances and non-empty interior}

For a compact set $E\subset\R^d$, $d\geq 2$, the Falconer distance problem asks for a dimensional threshold ensuring that the distance set
\begin{equation*}
\Delta(E) =\{|x-y|:x,y\in E\}
\end{equation*}
has positive Lebesgue measure. Falconer proved this with $\dimH(E)>(d+1)/2$ and conjectured that the sharp threshold is $d/2$ \cite{Falconer}. Subsequent work, including \cite{Bourgain,Wolff,ErdoganDistance,DuZhang,DuGuthOuWangWilsonZhang,DuIosevichOuWangZhang}, lowered the threshold culminating in $\frac{5}{4}$ when $d=2$ \cite{GuthIosevichOuWang} and $\frac{d}{2}+\frac{1}{4}-\frac{1}{8d+4}$ when $d\geq 3$ \cite{DuOuRenZhang}.

The pinned distance set at $y\in E$ is
\begin{equation*}
\Delta^y(E) =\{|x-y|:x\in E\}.
\end{equation*}
Peres and Schlag \cite{PeresSchlag} obtained a threshold $\dimH(E)>(d+1)/2$ that guarantees the existence of $y\in E$ such that $\Delta^y(E)$ has positive Lebesgue measure. We note this is a far stronger statement that immediately implies results for the original Falconer distance problem. Liu connected the proof techniques of the pinned and the original Falconer distance problem \cite{LiuPinned} which led to the best thresholds stated above for the original problem actually having been proven for the pinned variant. We note that the proofs for the best thresholds guarantee far more than just the existence of a single pin, but rather show the conclusion holds for $\mu$ a.e. $y\in E$ where $\mu$ is a Frostman measure on $E$.

Another variant of the original Falconer distance problem asks for non-empty interior instead of merely showing that $\Delta(E)$ has positive Lebesgue measure. Mattila and Sj\"olin obtained such a result at the dimensional threshold $(d+1)/2$ \cite{MattilaSjolin}. No improvements have been made on that threshold since, although progress has been made in different settings or under different assumptions. Combining the two previous variants leads to an even stronger statement, namely trying to show that the pinned distance set $\Delta^y(E)$ has non-empty interior. Peres and Schlag \cite{PeresSchlag} obtained the threshold $(d+2)/2$, $d\geq 3$, for such a statement. This has recently been improved by Borges, Foster, Ou and the first named author \cite{BorgesFosterOuPalsson}, who obtained the first non-trivial threshold of $7/4$ when $d=2$ and the improved threshold $12/5$ when $d=3$. In higher dimensions this was recently further improved by Borges, Ou and Pasquariello \cite{BorgesOuPasquariello} who obtained the threshold $\frac{d}{2}+\frac{3}{4}+\frac{3}{4(2d-1)}$ when $d\geq 4$.

\subsection{Scalar configuration maps}

Let $\Phi$ be a smooth real-valued function on an open subset of $\R^d\times\R^d$, and define
\begin{equation*}
\Delta_\Phi^y(E)=\{\Phi(x,y):x\in E\}.
\end{equation*}
Greenleaf, Iosevich and Taylor proved an unpinned nonempty-interior theorem for two-point configuration sets under uniform Sobolev smoothing assumptions for the associated generalized Radon transforms \cite[Theorem~1.1]{GIT19}. In the scalar nondegenerate case, the transforms have smoothing of order $(d-1)/2$ and their theorem gives the threshold $(d+1)/2$. The same fixed-parameter estimate used in their proof yields the following pinned positive-measure conclusion by the argument proved in Section~\ref{sec:scalar-positive}. Their work extended and built on earlier work such as \cite{EswarathasanIosevichTaylor,IosevichMourgoglouTaylor,ErdoganHartIosevich}.

\begin{corollary}[Pinned consequence of the Greenleaf--Iosevich--Taylor estimate in \cite{GIT19}]\label{cor:intro-pinned-positive}
Let $d\geq2$, let $E\subset\R^d$ be compact, and let $\mu$ be a probability measure supported on $E$ with finite $s$-energy. Assume that there is a countable collection of regular patches whose union contains $E\times E$ up to a set of $\mu\times\mu$ measure zero, and that on each patch the localized generalized Radon transforms associated with
\begin{equation*}
\Phi(x,y)=t
\end{equation*}
are nondegenerate Fourier integral operators of smoothing order $(d-1)/2$, uniformly for $t$ in compact parameter intervals. If
\begin{equation*}
s>\frac{d+1}{2},
\end{equation*}
then $\Delta_\Phi^y(E)$ has positive Lebesgue measure for $\mu$-almost every $y\in E$.
\end{corollary}

Our first result extends the pinned conclusion from positive Lebesgue measure to nonempty interior.

\begin{theorem}[Pinned nonempty interior for scalar configurations]\label{thm:intro-pinned-interior}
Let $d\geq3$, and assume the remaining hypotheses of Corollary~\ref{cor:intro-pinned-positive}. If
\begin{equation*}
s>\frac{d+2}{2},
\end{equation*}
then $\Delta_\Phi^y(E)$ has nonempty interior for $\mu$-almost every $y\in E$.
\end{theorem}

The restriction $d\geq3$ only removes the vacuous planar case, since a finite-energy measure supported on a subset of $\R^2$ cannot satisfy $s>2$. The proof uses the fact that differentiation in the level parameter preserves the canonical relation and raises the order of the Fourier integral operator by one. The resulting Sobolev regularity in $t$ gives a continuous pinned density.

\begin{corollary}[Pinned generalized norm distances]\label{cor:intro-generalized-distances}
Let $d\geq3$, let $B\subset\R^d$ be a symmetric bounded convex body with smooth boundary and everywhere nonvanishing Gaussian curvature, and let
\begin{equation*}
\Delta_B^y(E)=\{\|x-y\|_B:x\in E\}.
\end{equation*}
Let $\mu$ be a probability measure supported on a compact set $E\subset\R^d$ with finite $s$-energy. If
\begin{equation*}
s>\frac{d+2}{2},
\end{equation*}
then $\Delta_B^y(E)$ has nonempty interior for $\mu$-almost every $y\in E$.
\end{corollary}

Iosevich, Mourgoglou and Taylor proved that the unpinned generalized distance set contains an interval when $\dimH(E)>(d+1)/2$ \cite{IosevichMourgoglouTaylor}; this is also Corollary~1.2 of \cite{GIT19}. The pinned framework of Iosevich, Taylor and Uriarte-Tuero gives positive Lebesgue measure for $\mu$ a.e. $y\in E$ at the same threshold \cite{IosevichTaylorUriarteTuero}. Corollary~\ref{cor:intro-generalized-distances} gives the corresponding $\mu$-a.e. pinned non-empty interior conclusion at the higher threshold $\dimH(E)>(d+2)/2$. For the Euclidean norm, $(d+2)/2$ is the classical pinned non-empty interior threshold from Peres and Schlag \cite{PeresSchlag}.

\begin{corollary}[Variable-coefficient and Riemannian distances]\label{cor:intro-variable-distances}
Let $d\geq3$, let $E\subset\R^d$ be compact, and let $\mu$ be a probability measure supported on $E$ with finite $s$-energy. Let $\phi$ be smooth on a countable collection of open subsets whose union contains $E\times E$ up to a set of $\mu\times\mu$ measure zero. Assume that on each such subset
\begin{equation*}
\nabla_x\phi(x,y)\neq0,
\qquad
\nabla_y\phi(x,y)\neq0,
\end{equation*}
and that the Phong--Stein rotational-curvature condition holds on the level sets under consideration. If
\begin{equation*}
s>\frac{d+2}{2},
\end{equation*}
then, for $\mu$-almost every $y\in E$, the set
\begin{equation*}
\Delta_\phi^y(E)=\{\phi(x,y):x\in E\}
\end{equation*}
has nonempty interior.

In particular, let $(M,g)$ be a compact $d$-dimensional Riemannian manifold without boundary, and let $\rho_g$ denote the Riemannian distance induced by $g$. Let $E\subset M$ be compact, and let $\mu$ be a probability measure supported on $E$ with finite $s$-energy. Suppose that $E\times E$ is covered up to a $\mu\times\mu$ null set by countably many coordinate patches on which $\rho_g$ is smooth and the associated spherical mean operators have canonical graph relations. If $s>(d+2)/2$, then
\begin{equation*}
\Delta_{\rho_g}^y(E)
=
\{\rho_g(x,y):x\in E\}
\end{equation*}
has nonempty interior for $\mu$-almost every $y\in E$.
\end{corollary}

The variable-coefficient and Riemannian pinned results of Iosevich, Taylor and Uriarte-Tuero give positive Lebesgue measure at the threshold $(d+1)/2$ and also control the exceptional set of pins \cite{IosevichTaylorUriarteTuero}. Corollary~\ref{cor:intro-variable-distances} gives the corresponding nonempty-interior conclusion after the additional half-dimensional hypothesis. In the unpinned setting, the local Riemannian non-empty interior theorem appears as Corollary~1.3 of \cite{GIT19}.

\begin{corollary}[Pinned dot products and bilinear forms]\label{cor:intro-dot-products}
Let $d\geq3$, let $A$ be an invertible real $d\times d$ matrix, and let $\mu$ be a probability measure supported on a compact set $E\subset\R^d$ with finite $s$-energy. Assume that
\begin{equation*}
(\mu\times\mu)\big(\{(x,y)\in E\times E:x^TAy=0\}\big)=0.
\end{equation*}
If
\begin{equation*}
s>\frac{d+2}{2},
\end{equation*}
then, for $\mu$-almost every $y\in E$, the pinned bilinear-form set
\begin{equation*}
\{x^TAy:x\in E\}
\end{equation*}
has nonempty interior.
\end{corollary}

For $A=I$, this gives a pinned dot-product statement for $\Phi(x,y)=x\cdot y$. In that case this is already a consequence of the work of Bright, Marshall and Senger, see Theorem 1.2 (3) in \cite{BrightMarshallSenger}. 

\subsection{Volumes of simplices}

For $x_0,\ldots,x_k\in\R^d$, $2\leq k\leq d$, write
\begin{equation*}
\operatorname{Vol}_k(x_0,\ldots,x_k)
=
|(x_1-x_0)\wedge\cdots\wedge(x_k-x_0)|.
\end{equation*}
This is $k!$ times the Euclidean volume of the corresponding $k$-simplex in $\mathbb{R}^d$. A classic variant of the Falconer distance problem asks for a compact set $E\subset\mathbb{R}^d$, $d\geq 2$, how large $\dimH(E)$ needs to be to guarantee that the volume set
\begin{equation*}
\Vol_k(E)
=
\big\{
|(x_1-x_0)\wedge\cdots\wedge(x_k-x_0)|:
x_0,x_1,\ldots,x_k\in E
\big\}
\end{equation*}
has positive Lebesgue measure.

Initially, the focus was on the case $k=d$, where the $k$-simplex is in its natural ambient space. The first couple of results established the thresholds $\frac{3}{2}$ for areas of triangles in $\mathbb{R}^2$ \cite{ErdoganHartIosevich} and $\frac{8}{3}$ for volumes of tetrahedra in $\mathbb{R}^3$ \cite{GreenleafIosevichMourgoglou}. In \cite{GrafakosGreenleafIosevichPalsson} Grafakos, Greenleaf, Iosevich and the first listed author of this paper obtained the dimensional thresholds $d-1+\frac{1}{2d}$ if $d$ is even and $d-1+\frac{1}{2(d-1)}$ if $d$ is odd. They further conjectured that the threshold $d-1$ is sharp, as is clearly necessary, as seen by considering a hyperplane. This conjecture was solved by Shmerkin and Yavicoli \cite{ShmerkinYavicoli} who obtained a multiple pin version of the statement where they could pin a whole face of the simplex. For non-empty interior Greenleaf, Iosevich and Taylor in \cite{GIT22} obtained the threshold $\frac{5}{3}$ in $\mathbb{R}^2$ and $d-1+\frac{1}{d}$ in $\mathbb{R}^d$ when $d\geq 3$. Greenleaf, Iosevich and Taylor actually obtained a stronger statement. For a prescribed vertex $x_0\in\R^d$, define the strongly pinned $k$-volume set
\begin{equation*}
\Vol_k^{x_0}(E)
=
\big\{
|(x_1-x_0)\wedge\cdots\wedge(x_k-x_0)|:
x_1,\ldots,x_k\in E
\big\}.
\end{equation*}
Greenleaf, Iosevich and Taylor in \cite{GIT22} introduced the term ``strongly pinned'' to emphasize that the pin is prescribed in advance so any result must hold for all $x_0\in\R^d$. Their theorem gives non-empty interior of $\Vol_d^{x_0}(E)$ for every $x_0\in\R^d$ when $\dimH(E)>d-1+\frac{1}{d}$, $d\geq 3$.

With the resolution of the conjecture by Shmerkin and Yavicoli when $k=d$ the attention has now turned to the case when $k<d$. Gaitan Montejo and the first named author \cite{GaitanPalsson2026} obtained the threshold $\max\{k-1,d-k+1\}$, which is sharp if $k\leq d\leq 2k-2$, as can be seen by considering a $k-1$ dimensional plane and reproduces the result of Shmerkin and Yavicoli when $k=d$.  Borges, Foster, Ou, Romero Acosta and the first named author of this paper \cite{BorgesFosterOuPalssonRomeroAcosta} developed a theory of volume vectors associated with hypergraphs of simplices, extending some earlier work in that direction such as \cite{McDonaldAreas,GaloMcDonald,GIT24,GIT25}. As a byproduct of their methods, they obtained the dimensional threshold $\frac{d+k}{2}$ for guaranteeing that $\Vol_k(E)$ has positive Lebesgue measure even with a whole face pinned. This built on their earlier work on distance graphs \cite{BorgesFosterOuPalssonRomeroAcostaDistGraph}. Using the very recent paper \cite{BorgesOuPasquariello} as input, this dimensional threshold can be further improved to $\frac{d+k-1}{2}+\frac{1}{4}+\frac{2k+1}{4(2d+1)}$ when $k<d$.

Motivated by the multilinear generalized Radon framework of Grafakos, Greenleaf, Iosevich and the first named author \cite{GrafakosGreenleafIosevichPalsson}, we use a different geometric principle. Once one vector is fixed, a $k$-volume factors through orthogonal projection onto its perpendicular hyperplane. Iterating this observation reduces both positive-measure and nonempty-interior questions to a two-vector area theorem. The energy argument behind the reduction is closely related to the classical projection theory of Marstrand and Mattila \cite{Marstrand,MattilaBook}.

The base case for our results is the following strongly pinned non-empty interior result for triangle areas in any dimension.

\begin{theorem}[Strongly pinned triangle areas]\label{thm:area}
Let $d\geq2$, let $E\subset\R^d$ be compact, and let $x_0\in\R^d$. If
\begin{equation*}
\dimH(E)>\frac{d+1}{2},
\end{equation*}
then
\begin{equation*}
\Vol_2^{x_0}(E)
=
\big\{|(x_1-x_0)\wedge(x_2-x_0)|:x_1,x_2\in E\big\}
\end{equation*}
has nonempty interior.
\end{theorem}

When $d=2$ this improves on the threshold $\frac{5}{3}$ obtained by Greenleaf, Iosevich and Taylor in \cite{GIT22} to establish non-empty interior, while additionally including a strong pin. While weaker than the result of Shmerkin and Yavicoli for establishing merely positive Lebesgue measure when $d=2$, it improves in the single pin setting on the threshold $\frac{d+1}{2}+\frac{d+3}{4d+2}$, $d\geq 3$, obtained as a consequence of \cite{BorgesFosterOuPalssonRomeroAcosta,BorgesOuPasquariello}. The key observation is that the fiber obtained after fixing $x_2-x_0=r\theta$ is the cylinder $|P_{\theta^\perp}(x_1-x_0)|=t/r$. Its spherical cross-section and its free axial direction together yield smoothing of order $(d-1)/2$. We thus refer to this as a cylinder estimate for triangles.

The threshold $\frac{d+1}{2}+\frac{d+3}{4d+2}$, $d\geq 3$, for areas of triangles obtained as a consequence of \cite{BorgesFosterOuPalssonRomeroAcosta,BorgesOuPasquariello} is valid for two pins. For $y\neq x_0$, define
\begin{equation*}
\Vol_k^{x_0,y}(E)
=
\big\{
|(y-x_0)\wedge(x_1-x_0)\wedge\cdots\wedge(x_{k-1}-x_0)|:
x_1,\ldots,x_{k-1}\in E
\big\}.
\end{equation*}
With a very similar proof structure we are also able to obtain a two pin theorem.

\begin{theorem}[Doubly pinned triangle areas]\label{thm:double-pinned-area}
Let $E\subset\R^d$ be compact, fix $x_0\in\R^d$, and let $\mu$ be a probability measure supported on $E$ with finite $s$-energy. Then, for $\mu$-almost every $y\in E\setminus\{x_0\}$, the following conclusions hold:
\begin{enumerate}
\item If $d\geq2$ and
\begin{equation*}
s>\frac{d+1}{2},
\end{equation*}
then $\Vol_2^{x_0,y}(E)$ has positive Lebesgue measure.
\item If $d\geq3$ and
\begin{equation*}
s>\frac{d+2}{2},
\end{equation*}
then $\Vol_2^{x_0,y}(E)$ has nonempty interior.
\end{enumerate}
\end{theorem}

The first part of this theorem not only directly improves on the threshold $\frac{d+1}{2}+\frac{d+3}{4d+2}$, $d\geq 3$, from \cite{BorgesFosterOuPalssonRomeroAcosta,BorgesOuPasquariello}, but further improves on it by having one of the pins be strongly pinned instead of obtaining a $\mu\times\mu$-a.e. statement in the pins. The second part of the theorem further extends this to the non-empty interior setting with a slightly worse threshold. We note that when $d=2$, the positive-measure statement can also be viewed as a pinned dot-product estimate through $|u\wedge v|=|u\cdot v^\perp|$.

We distinguish between two types of double pinning. We call a conclusion \textit{doubly pinned} if one vertex $x_0$ is prescribed and the conclusion holds for almost every choice of the second vertex $y$. 
We call it \textit{doubly strongly pinned} if both $x_0$ and $y$ may be prescribed in advance and the conclusion holds for every admissible ordered pair $(x_0,y)$ with $y\neq x_0$. Thus the adjective ``strongly'' means that the pins can be chosen in advance, not simply that the corresponding vertices are held fixed once the configuration is defined.

Instead of extending the Fourier analytic arguments to $k$-simplices we instead use geometric arguments to reduce the arguments down to triangles. Next is our main geometric reduction theorem.

\begin{theorem}[Geometric reduction to triangle areas]\label{thm:geometric-reduction}
Suppose that, for every $m\geq2$, there are thresholds $X_{\mathrm{pm}}(m)$ and $X_{\mathrm{int}}(m)$ with the following properties: for every compact set $F\subset\R^m$ and every $z_0\in\R^m$,
\begin{enumerate}
\item if
\begin{equation*}
\dimH(F)>X_{\mathrm{pm}}(m),
\end{equation*}
then the strongly pinned triangle-area set $\Vol_2^{z_0}(F)$ has positive Lebesgue measure;
\item if
\begin{equation*}
\dimH(F)>X_{\mathrm{int}}(m),
\end{equation*}
then $\Vol_2^{z_0}(F)$ has nonempty interior.
\end{enumerate}
Then, for every $3\leq k\leq d$, every compact set $E\subset\R^d$, every $x_0\in\R^d$, and every $y\neq x_0$,
\begin{enumerate}
\item if
\begin{equation*}
\dimH(E)>X_{\mathrm{pm}}(d-k+2)+k-2,
\end{equation*}
then $\Vol_k^{x_0,y}(E)$ has positive Lebesgue measure;
\item if
\begin{equation*}
\dimH(E)>X_{\mathrm{int}}(d-k+2)+k-2,
\end{equation*}
then $\Vol_k^{x_0,y}(E)$ has nonempty interior.
\end{enumerate}
\end{theorem}

The reduction is doubly strongly pinned in its first two vertices: both the base point $x_0$ and the second vertex $y$ may be prescribed in advance. The directions selected during the projection induction affect only the remaining freely varying vertices.

Combining the geometric reduction with the strongly pinned triangle theorem gives the main simplex-volume theorem.

\begin{theorem}[Doubly strongly pinned simplex volumes]\label{thm:double-pinned-volumes}
Let $3\leq k\leq d$, let $E\subset\R^d$ be compact, and fix $x_0\in\R^d$. If
\begin{equation*}
\dimH(E)>\frac{d+k-1}{2},
\end{equation*}
then, for every prescribed $y\in E\setminus\{x_0\}$, the set $\Vol_k^{x_0,y}(E)$ has nonempty interior.
\end{theorem}

It is important that Theorem~\ref{thm:double-pinned-volumes} uses the strongly pinned triangle theorem, not the doubly pinned triangle theorem. For example, when $k=3$, translating $x_0=0$ and setting $F=P_{y^\perp}E$ gives the exact identity
\begin{equation*}
\Vol_3^{0,y}(E)=|y|\,\Vol_2^0(F).
\end{equation*}
Although the tetrahedron is doubly pinned, the projected triangle has only the origin prescribed; its other two vertices vary freely over $F$. Theorem~\ref{thm:area} in the ambient space $y^\perp\simeq\R^{d-1}$ therefore requires $\dimH(F)>d/2$. Since $\dimH(F)\geq\dimH(E)-1$, this gives the tetrahedral threshold $(d+2)/2$. A genuinely doubly pinned triangle would arise only after prescribing a third tetrahedral vertex, in which case the nonempty-interior threshold would become $(d+3)/2$.

Since $\Vol_k^{x_0,y}(E)\subset\Vol_k^{x_0}(E)$, the second pin may be forgotten without changing the dimensional threshold.

\begin{corollary}[Strongly pinned lower-rank volumes]\label{cor:strongly-pinned-volumes}
Let $2\leq k<d$, let $E\subset\R^d$ be compact, and let $x_0\in\R^d$. If
\begin{equation*}
\dimH(E)>\frac{d+k-1}{2},
\end{equation*}
then $\Vol_k^{x_0}(E)$ has nonempty interior.
\end{corollary}

At the full rank endpoint $k=d$, Theorem~\ref{thm:double-pinned-volumes} gives the genuinely doubly pinned conclusion at $\dimH(E)>d-1/2$. If the second pin is forgotten, the direct strongly pinned theorem of Greenleaf, Iosevich and Taylor is sharper: it gives nonempty interior for every prescribed $x_0$ under $\dimH(E)>d-1+1/d$ \cite{GIT22}. Applying one projection step and then their theorem in dimension $d-1$ would instead require $\dimH(E)>d-1+1/(d-1)$, so the direct full-rank theorem is preferable. Thus, at $k=d$, both results give nonempty interior: the Greenleaf--Iosevich--Taylor theorem has the sharper threshold for one prescribed pin, while Theorem~\ref{thm:double-pinned-volumes} retains a second prescribed vertex. In the proper lower-rank range, when $k<d$ the projection theorem yields the threshold $\frac{d+k-1}{2}$, which for two pins improves on the threshold $\frac{d+k-1}{2}+\frac{1}{4}+\frac{2k+1}{4(2d+1)}$ from \cite{BorgesFosterOuPalssonRomeroAcosta,BorgesOuPasquariello}, but we note that the latter can allow for many more pins (although they are not strongly pinned). When $d>3k-3$ this also improves in the two pin setting on the threshold $d-k+1$ stated in \cite{GaitanPalsson2026}, which is just obtained by slicing and repeated use of the result by Shmerkin and Yavicoli. When $2k-2< d \leq 3k-3$ the threshold $d-k+1$ still remains the best for a positive Lebesgue measure conclusion and allows for multiple pins.

\subsection{Organization}

Section~\ref{sec:radon-preliminaries} collects the generalized Radon-transform notation, the level-parameter regularization, the parameter-derivative estimate, and the Sobolev-valued function spaces used below. Section~\ref{sec:scalar-positive} develops the scalar pinned-measure principle, proves both the positive-measure and nonempty-interior conclusions, and derives Corollary~\ref{cor:intro-pinned-positive} and Theorem~\ref{thm:intro-pinned-interior}. Section~\ref{sec:scalar-examples} verifies the hypotheses of this principle for generalized norm distances, variable-coefficient and Riemannian distances, and dot products and bilinear forms. Section~\ref{sec:geometric-reduction} proves the geometric reduction statements. Section~\ref{sec:triangle-area} proves the cylinder estimate and the strongly and doubly pinned triangle theorems. Section~\ref{sec:simplex-consequences} derives the simplex-volume results. Section~\ref{sec:future} records future directions. 

\section{Preliminaries on generalized Radon transforms}\label{sec:radon-preliminaries}

Let $X\subset\R^{d_X}$ and $Y\subset\R^{d_Y}$ be open, let $\Phi:X\times Y\to\R$ be smooth, and let $\chi\in C_c^\infty(X\times Y)$ be nonnegative and supported on a compact regular patch for the level relation $\Phi(x,y)=t$, on which $D_x\Phi$ and $D_y\Phi$ do not vanish. Choose compact intervals $J,J^*\subset\R$ such that
\begin{equation*}
\Phi(\supp\chi)\subset\operatorname{int}(J)
\subset J
\subset\operatorname{int}(J^*).
\end{equation*}
The associated localized generalized Radon transform is
\begin{equation*}
T_t^\chi f(y)
=
\int f(x)\chi(x,y)\delta(\Phi(x,y)-t)\,dx.
\end{equation*}
The delta notation is understood in the Leray-measure sense on the regular level set.

\subsection{Level-parameter regularization}

Let $\rho\in C_c^\infty(\R)$ be nonnegative, even, supported in $(-1,1)$, and normalized by $\int\rho=1$. Set
\begin{equation*}
\rho_\eps(u)=\eps^{-1}\rho(u/\eps)
\end{equation*}
and define
\begin{equation*}
T_{t,\eps}^\chi f(y)
=
\int\rho_\eps(\tau-t)T_\tau^\chi f(y)\,d\tau.
\end{equation*}
Equivalently,
\begin{equation*}
T_{t,\eps}^\chi f(y)
=
\int f(x)\chi(x,y)\rho_\eps(\Phi(x,y)-t)\,dx.
\end{equation*}
For a finite measure $\mu$ on $X$, write
\begin{equation*}
h_\eps(y,t)
=
T_{t,\eps}^\chi\mu(y)
=
\int\chi(x,y)\rho_\eps(\Phi(x,y)-t)\,d\mu(x).
\end{equation*}
Thus $h_\eps(y,\cdot)$ is the regularized density of the localized pinned pushforward $\nu_y^\chi$, defined by
\begin{equation*}
\int\varphi(t)\,d\nu_y^\chi(t)
=
\int\chi(x,y)\varphi(\Phi(x,y))\,d\mu(x).
\end{equation*}
We regularize the level parameter rather than the Frostman measure. Since $\|\rho_\eps\|_1=1$, every operator estimate for $T_\tau^\chi$ that is uniform for $\tau\in J^*$ passes directly to $T_{t,\eps}^\chi$ for $t\in J$ and sufficiently small $\eps$. Once uniform parameter-derivative estimates are available, integration by parts in $\tau$ shows that they also pass to the regularized family. Consequently, the corresponding parameter-difference estimates are uniform in $\eps$.

\subsection{Finite-energy and Frostman trace estimates}

For a finite Borel measure $\sigma$ on $\R^n$ and $0<s<n$, write
\begin{equation*}
I_s(\sigma)
=
\iint |x-y|^{-s}\,d\sigma(x)d\sigma(y).
\end{equation*}

\begin{lemma}[Finite-energy and Frostman estimates]\label{lem:energy-trace}
Let $0<s<n$.

\begin{enumerate}
\item If $I_s(\sigma)<\infty$ and $P_i$ is a Littlewood--Paley projection to frequencies of size $2^i$, then
\begin{equation*}
\|P_i\sigma\|_{L^2(\R^n)}
\lesssim
2^{i(n-s)/2}I_s(\sigma)^{1/2}.
\end{equation*}

\item Suppose that $\sigma$ is $s$-Frostman, so that
\begin{equation*}
\sigma(B(z,r))\leq C_\sigma r^s
\end{equation*}
for every $z\in\R^n$ and $r>0$. If $\widehat g$ is supported where $|\xi|\lesssim2^i$, then
\begin{equation*}
\|g\|_{L^2(d\sigma)}
\lesssim
C_\sigma^{1/2}2^{i(n-s)/2}\|g\|_{L^2(\R^n)}.
\end{equation*}

\item If $I_s(\sigma)<\infty$, then there are increasing Borel sets $E_N$ whose union has full $\sigma$-measure and such that each restriction $\sigma|_{E_N}$ is $s$-Frostman.
\end{enumerate}
\end{lemma}

\begin{proof}
The first estimate follows from the Fourier representation of Riesz energy:
\begin{equation*}
I_s(\sigma)
\approx
\int_{\R^n}
|\widehat\sigma(\xi)|^2|\xi|^{s-n}\,d\xi.
\end{equation*}

For the second estimate, choose a Schwartz function $\psi_i$ at scale $2^{-i}$ such that $g=g*\psi_i$. Cauchy--Schwarz gives
\begin{equation*}
|g(x)|^2
\lesssim
\int |g(z)|^2|\psi_i(x-z)|\,dz.
\end{equation*}
After integration against $\sigma$ and reversal of the order of integration,
\begin{equation*}
\int |g(x)|^2\,d\sigma(x)
\lesssim
\int |g(z)|^2(|\psi_i|*\sigma)(z)\,dz.
\end{equation*}
The Frostman condition and the rapid decay of $\psi$ imply
\begin{equation*}
\sup_z(|\psi_i|*\sigma)(z)
\lesssim
C_\sigma 2^{i(n-s)},
\end{equation*}
which proves the claim.

For the final assertion, let
\begin{equation*}
U_s^\sigma(x)
=
\int |x-y|^{-s}\,d\sigma(y)
\end{equation*}
and set
\begin{equation*}
E_N=\{x:U_s^\sigma(x)\leq N\}.
\end{equation*}
Since $I_s(\sigma)<\infty$, the union of the sets $E_N$ has full $\sigma$-measure. If $E_N\cap B(z,r)\neq\varnothing$, choose $x\in E_N\cap B(z,r)$. Then
\begin{equation*}
\sigma|_{E_N}(B(z,r))
\leq
\sigma(B(z,r))
\leq
(2r)^sU_s^\sigma(x)
\leq
2^sNr^s.
\end{equation*}
Thus $\sigma|_{E_N}$ is $s$-Frostman.
\end{proof}

\subsection{Parameter derivatives}

For smooth one-parameter families of scalar generalized Radon transforms, differentiation in the level parameter preserves the canonical relation and raises the order by at most one. 

Throughout this section, a Fourier-integral estimate that is uniform for $t$ in a compact interval is understood to have constants uniform when the localized amplitude ranges over a bounded subset of the relevant symbol class.

\begin{lemma}[Parameter derivatives for scalar Radon families]\label{lem:scalar-parameter-derivative}
Let $T_t^\chi$ be as above. Suppose that, for some $\gamma\geq0$,
\begin{equation}\label{eq:radon-fixed-estimate}
\|T_t^\chi f_i\|_{L^2(Y)}
\lesssim
2^{-\gamma i}\|f_i\|_{L^2(X)},
\end{equation}
where $f_i=P_i^Xf$. Suppose that \eqref{eq:radon-fixed-estimate} holds uniformly for $t\in J^*$ in the preceding sense. Then, for every integer $N\geq0$,
\begin{equation}\label{eq:higher-parameter-derivatives}
\|\partial_t^NT_t^\chi f_i\|_{L^2(Y)}
\lesssim_N
2^{Ni}2^{-\gamma i}\|f_i\|_{L^2(X)}
\end{equation}
uniformly for $t\in J^*$.

The same estimates hold uniformly for the regularized operators:
\begin{equation}\label{eq:regularized-parameter-derivatives}
\|\partial_t^NT_{t,\eps}^\chi f_i\|_{L^2(Y)}
\lesssim_N
2^{Ni}2^{-\gamma i}\|f_i\|_{L^2(X)}
\end{equation}
for $t\in J$ and sufficiently small $\eps$. In particular,
\begin{equation}\label{eq:regularized-parameter-difference}
\|(T_{t,\eps}^\chi-T_{t',\eps}^\chi)f_i\|_{L^2(Y)}
\lesssim
\min\{1,2^i|t-t'|\}\,2^{-\gamma i}\|f_i\|_{L^2(X)}
\end{equation}
uniformly for $t,t'\in J$ and sufficiently small $\eps$.
\end{lemma}

\begin{proof}
On the chosen patch, the Schwartz kernel has the oscillatory representation
\begin{equation*}
K_t(y,x)
=
\int_{\R}
e^{2\pi i\lambda(\Phi(x,y)-t)}
a(x,y,t,\lambda)\,d\lambda,
\end{equation*}
where $a$ is a classical symbol. Differentiating $N$ times in $t$ produces a finite sum of terms whose amplitudes have the form
\begin{equation*}
\lambda^\ell\partial_t^{N-\ell}a(x,y,t,\lambda),
\qquad
0\leq\ell\leq N.
\end{equation*}
The canonical relation is unchanged. On the microlocal support of the operator acting on an input at frequency $2^i$, regularity of the phase gives
\begin{equation*}
|\lambda|\approx|\xi|\approx2^i.
\end{equation*}
Thus each $t$-derivative raises the order of the amplitude by at most one and costs at most a factor $2^i$. After factoring out $2^{Ni}$, the resulting amplitudes remain in a bounded subset of the symbol class to which the fixed-parameter estimate applies. This proves \eqref{eq:higher-parameter-derivatives}.

For the regularized family,
\begin{equation*}
T_{t,\eps}^\chi
=
\int\rho_\eps(\tau-t)T_\tau^\chi\,d\tau.
\end{equation*}
The estimate with $N=0$ follows from Minkowski's inequality and $\|\rho_\eps\|_1=1$. For $N\geq1$, integration by parts in $\tau$ gives
\begin{equation*}
\partial_t^NT_{t,\eps}^\chi
=
\int\rho_\eps(\tau-t)\partial_\tau^NT_\tau^\chi\,d\tau.
\end{equation*}
Estimate \eqref{eq:regularized-parameter-derivatives} therefore follows from \eqref{eq:higher-parameter-derivatives}. Finally, the fixed estimate gives
\begin{equation*}
\|(T_{t,\eps}^\chi-T_{t',\eps}^\chi)f_i\|_{L^2(Y)}
\lesssim
2^{-\gamma i}\|f_i\|_{L^2(X)},
\end{equation*}
while the fundamental theorem of calculus and the estimate with $N=1$ give
\begin{equation*}
\|(T_{t,\eps}^\chi-T_{t',\eps}^\chi)f_i\|_{L^2(Y)}
\lesssim
2^i|t-t'|\,2^{-\gamma i}\|f_i\|_{L^2(X)}.
\end{equation*}
Taking the better of these two estimates proves \eqref{eq:regularized-parameter-difference}.
\end{proof}

\subsection{Input--output frequency compatibility}

The fixed-parameter norm estimate alone does not imply that input and output frequencies are comparable. For the localized generalized Radon transforms considered here, this compatibility follows from the canonical relation.

\begin{lemma}[Input--output frequency compatibility]\label{lem:scalar-frequency-compatibility}
Let $P_i^X$ and $P_j^Y$ denote Littlewood--Paley projections in the input and output variables. There exists an integer $C_0\geq1$ such that, for every $M\geq0$,
\begin{equation}\label{eq:scalar-frequency-compatibility}
\|P_j^YT_t^\chi P_i^X\|_{L^2(X)\to L^2(Y)}
\lesssim_M
2^{-M\max\{i,j\}}
\end{equation}
whenever $|i-j|>C_0$, uniformly for $t\in J^*$. The same conclusion holds for parameter derivatives and for the regularized operators $T_{t,\eps}^\chi$.
\end{lemma}

\begin{proof}
The canonical relation of $T_t^\chi$ is contained in
\begin{equation*}
\mathcal C_t
=
\big\{
(y,\tau D_y\Phi(x,y);x,-\tau D_x\Phi(x,y)):
\Phi(x,y)=t,\ \tau\neq0
\big\}.
\end{equation*}
Since the patch is compact and $D_x\Phi$ and $D_y\Phi$ do not vanish there, there are constants $c,C>0$ such that
\begin{equation*}
c|\tau|
\leq
|\tau D_x\Phi(x,y)|,\,
|\tau D_y\Phi(x,y)|
\leq
C|\tau|
\end{equation*}
throughout the patch. Thus the input and output covectors have comparable magnitudes on $\mathcal C_t$.

If $|i-j|$ is sufficiently large, the frequency supports of $P_i^X$ and $P_j^Y$ do not meet the canonical relation. The corresponding oscillatory integral therefore has no stationary point, and repeated integration by parts gives \eqref{eq:scalar-frequency-compatibility}. Differentiation in $t$ only modifies the amplitude and does not change the canonical relation. Averaging in the level parameter gives the same conclusion for $T_{t,\eps}^\chi$.
\end{proof}

\subsection{Sobolev-valued function spaces}

For $0<\beta<1$, $H^\beta(J)$ denotes the fractional Sobolev space in the parameter $t$, with norm equivalent to
\begin{equation*}
\|g\|_{L^2(J)}^2
+
\iint_{J\times J}
\frac{|g(t)-g(t')|^2}{|t-t'|^{1+2\beta}}\,dt\,dt'.
\end{equation*}
If $B$ is a Banach space, $L^2(d\lambda;B)$ denotes the Bochner space of strongly measurable $B$-valued functions $F$ satisfying
\begin{equation*}
\int\|F(y)\|_B^2\,d\lambda(y)<\infty;
\end{equation*}
see \cite[Chapter~1]{HytonenVanNeervenVeraarWeis}. In particular, $L^2(d\lambda(y);H^\beta(J))$ consists of strongly measurable families $h(y,\cdot)\in H^\beta(J)$ satisfying
\begin{equation*}
\int\|h(y,\cdot)\|_{H^\beta(J)}^2\,d\lambda(y)<\infty.
\end{equation*}
When $\beta>\frac{1}{2}$, the one-dimensional scalar Sobolev embedding gives
\begin{equation*}
H^\beta(J)\hookrightarrow C^0(J);
\end{equation*}
see \cite[Section~8, Theorem 8.2]{DiNezzaPalatucciValdinoci}. We apply this scalar embedding fiberwise: membership in $L^2(d\lambda;H^\beta(J))$ implies that $h(y,\cdot)\in H^\beta(J)$ for $\lambda$-almost every $y$, and hence $h(y,\cdot)$ has a continuous representative for almost every $y$.

\section{Pinned scalar configuration measures and nonempty interior}\label{sec:scalar-positive}

With the notation and regularization of Section~\ref{sec:radon-preliminaries}, the following asymmetric principle contains both the fixed-parameter positive-measure argument and the nonempty-interior upgrade.

The proposition is stated for nonzero localized pushforwards. The full pinned pushforward is always nonzero, indeed it has total mass one because $\mu$ is a probability measure, and after the proof we show that, for almost every pin, at least one localized pushforward is nonzero.

\begin{proposition}[Scalar double-pinning principle]\label{prop:scalar-double-pinning}
Suppose that, for some $\gamma\geq0$, every Littlewood--Paley piece $f_i=P_i^Xf$ satisfies
\begin{equation}\label{eq:scalar-fixed-estimate}
\|T_t^\chi f_i\|_{L^2(Y)}
\lesssim
2^{-\gamma i}\|f_i\|_{L^2(X)}
\end{equation}
uniformly for $t\in J^*$. Assume that this Fourier-integral estimate has the symbol-class uniformity required in Lemma~\ref{lem:scalar-parameter-derivative}.

Let $\mu$ and $\lambda$ be compactly supported probability measures on $X$ and $Y$ with finite $s_X$-energy and finite $s_Y$-energy, respectively. Set
\begin{equation*}
\eta
=
\gamma+\frac{s_X-d_X}{2}+\frac{s_Y-d_Y}{2}.
\end{equation*}
Then the following hold.

If $\eta>0$, then, for $\lambda$-almost every $y$, every nonzero localized pinned pushforward $\nu_y^\chi$ has an $L^2$ density and therefore has support of positive Lebesgue measure.

If $\eta>\frac{1}{2}$, then, for $\lambda$-almost every $y$, every nonzero localized pinned pushforward $\nu_y^\chi$ has a continuous density and therefore has support containing a nonempty interval.
\end{proposition}

\begin{proof}

By Lemma~\ref{lem:energy-trace}, there are increasing Borel sets $Y_N\subset Y$ whose union has full $\lambda$-measure and such that each normalized restriction
\begin{equation*}
\lambda_N
=
\frac{\lambda|_{Y_N}}{\lambda(Y_N)}
\end{equation*}
is $s_Y$-Frostman whenever $\lambda(Y_N)>0$. Since the conclusion is an almost-everywhere statement in $y$, it is enough to prove it with $\lambda$ replaced by each $\lambda_N$. We may therefore assume throughout the proof that $\lambda$ is $s_Y$-Frostman.

By Lemma~\ref{lem:scalar-parameter-derivative},
\begin{equation}\label{eq:scalar-derivative-estimate}
\|\partial_tT_t^\chi f_i\|_{L^2(Y)}
\lesssim
2^i2^{-\gamma i}\|f_i\|_{L^2(X)}
\end{equation}
uniformly for $t\in J^*$. The same fixed-parameter and derivative estimates hold uniformly for the regularized operators, and \eqref{eq:regularized-parameter-difference} gives the corresponding parameter-difference estimate. Lemma~\ref{lem:scalar-frequency-compatibility} shows that an input at frequency $2^i$ produces output frequencies comparable to $2^i$, up to rapidly decaying errors.

Write $\mu_i=P_i^X\mu$. The output frequency compatibility and the Frostman trace estimate in the $Y$ variable, followed by \eqref{eq:scalar-fixed-estimate} and the finite-energy estimate for the input measure, give
\begin{align*}
\|T_{t,\eps}^\chi\mu_i\|_{L^2(d\lambda(y))}
&\lesssim
2^{i(d_Y-s_Y)/2}
\|T_{t,\eps}^\chi\mu_i\|_{L^2(Y)}
\\
&\lesssim
2^{i(d_Y-s_Y)/2}
2^{-\gamma i}
\|\mu_i\|_{L^2(X)}
\\
&\lesssim
2^{i(d_Y-s_Y)/2}
2^{-\gamma i}
2^{i(d_X-s_X)/2}
=
2^{-\eta i}.
\end{align*}
If $\eta>0$, summation over $i$ gives
\begin{equation*}
\sup_{0<\eps<\eps_0}
\|h_\eps\|_{L^2(d\lambda(y)\,dt)}
<\infty.
\end{equation*}
Passing to a subsequence, $h_\eps$ converges weakly in $L^2(d\lambda(y)\,dt)$ to a function $h(y,t)$.

For $\varphi\in C_c(J)$,
\begin{equation*}
\int_J\varphi(t)h_\eps(y,t)\,dt
=
\int\chi(x,y)(\varphi*\rho_\eps)(\Phi(x,y))\,d\mu(x).
\end{equation*}
Multiplying by a bounded measurable function $\psi(y)$, integrating with respect to $\lambda(y)$, and passing to the limit gives
\begin{equation*}
\int\psi(y)\int_J\varphi(t)h(y,t)\,dt\,d\lambda(y)
=
\int\psi(y)\int\chi(x,y)\varphi(\Phi(x,y))\,d\mu(x)\,d\lambda(y).
\end{equation*}
Using a countable dense family of test functions $\varphi$, we conclude that, for $\lambda$-almost every $y$, the measure $\nu_y^\chi$ is absolutely continuous on $J$ with density $h(y,\cdot)\in L^2(J)$. If $\nu_y^\chi$ is nonzero, then its density is nonzero in $L^2(J)$, and hence its support has positive Lebesgue measure.

Applying the input finite-energy estimate and the output Frostman trace estimate to \eqref{eq:regularized-parameter-difference} gives
\begin{equation*}
\|(T_{t,\eps}^\chi-T_{t',\eps}^\chi)\mu_i\|_{L^2(d\lambda(y))}
\lesssim
\min\{1,2^i|t-t'|\}\,2^{-\eta i}.
\end{equation*}
Hence, for every $0<\alpha<\min\{1,\eta\}$,
\begin{equation*}
\|h_\eps(\cdot,t)-h_\eps(\cdot,t')\|_{L^2(d\lambda)}
\lesssim
|t-t'|^\alpha
\end{equation*}
uniformly in $\eps$. If $\eta>\frac{1}{2}$, choose
\begin{equation*}
\frac{1}{2}<\beta<\alpha<\min\{1,\eta\}.
\end{equation*}
The Gagliardo characterization of $H^\beta(J)$ then gives
\begin{equation*}
\sup_{0<\eps<\eps_0}
\|h_\eps\|_{L^2(d\lambda(y);H^\beta(J))}
<\infty.
\end{equation*}
Passing to a further subsequence, $h_\eps$ converges weakly in this Bochner space to the same density $h$ identified above. By the definition of the Bochner space, $h(y,\cdot)\in H^\beta(J)$ for $\lambda$-almost every $y$. Since $\beta>\frac{1}{2}$, the scalar Sobolev embedding of Section~\ref{sec:radon-preliminaries} applies fiberwise and shows that, for $\lambda$-almost every $y$, the measure $\nu_y^\chi$ has a continuous density. Because $\chi$ is nonnegative, this continuous representative is nonnegative. If $\nu_y^\chi$ is nonzero, its density is positive at some point and therefore positive on a nonempty open interval. Since
\begin{equation*}
\supp\nu_y^\chi\subset\Delta_\Phi^y(\supp\mu),
\end{equation*}
the localized pinned configuration set contains a nonempty interval.
\end{proof}

\begin{corollary}[Equal-dimensional scalar configurations]\label{cor:equal-dimensional-scalar}
Suppose $X=Y\subset\R^d$, $\mu=\lambda$ has finite $s$-energy, and the fixed-parameter estimate \eqref{eq:scalar-fixed-estimate} holds with smoothing order $\gamma$ and with the symbol-class uniformity required in Lemma~\ref{lem:scalar-parameter-derivative}. Then every nonzero localized pinned pushforward has an $L^2$ density for $\mu$-almost every pin if
\begin{equation*}
s>d-\gamma,
\end{equation*}
and has a continuous density for $\mu$-almost every pin if
\begin{equation*}
s>d-\gamma+\frac{1}{2}.
\end{equation*}
\end{corollary}

We now verify that the nonzero hypothesis in Proposition~\ref{prop:scalar-double-pinning} is satisfied on at least one regular patch for almost every pin. For every $y\in Y$, the full pinned pushforward
\begin{equation*}
\nu_y=(\Phi(\cdot,y))_\#\mu
\end{equation*}
has total mass
\begin{equation*}
\nu_y(\R)=\mu(X)=1.
\end{equation*}
Let $U\subset X\times Y$ be the union of the regular patches and let $\{\chi_n\}_{n\geq1}$ be a countable locally finite nonnegative partition of unity on $U$, subordinate to these patches. Suppose
\begin{equation*}
(\mu\times\lambda)((\supp\mu\times\supp\lambda)\setminus U)=0.
\end{equation*}
By Fubini's theorem, for $\lambda$-almost every $y$, the section
\begin{equation*}
U_y=\{x\in\supp\mu:(x,y)\in U\}
\end{equation*}
has full $\mu$-measure. For every such $y$,
\begin{equation*}
\sum_{n\geq1}\nu_y^{\chi_n}(\R)
=
\int_{U_y}\sum_{n\geq1}\chi_n(x,y)\,d\mu(x)
=
1.
\end{equation*}
Therefore at least one localized pinned pushforward $\nu_y^{\chi_n}$ is nonzero. Since
\begin{equation*}
\supp\nu_y^{\chi_n}\subset\Delta_\Phi^y(\supp\mu),
\end{equation*}
a positive-measure or nonempty-interior conclusion for this localized measure gives the corresponding conclusion for the full pinned configuration set.

\begin{proof}[Proof of Corollary~\ref{cor:intro-pinned-positive}]
On each regular patch the nondegenerate Fourier-integral estimate has smoothing order $\gamma=(d-1)/2$ and is uniform for amplitudes in bounded subsets of the relevant symbol class. Apply Proposition~\ref{prop:scalar-double-pinning} with $X=Y=\R^d$ and $\mu=\lambda$. Then
\begin{equation*}
\eta
=
\frac{d-1}{2}+s-d
=
s-\frac{d+1}{2}.
\end{equation*}
The hypothesis $s>(d+1)/2$ gives $\eta>0$. Applying the proposition on each patch, intersecting the resulting countably many full-measure sets of pins, and using the preceding localization argument proves the conclusion.
\end{proof}

\begin{proof}[Proof of Theorem~\ref{thm:intro-pinned-interior}]
On every regular patch, the nondegenerate Fourier-integral estimate has smoothing order $\gamma=(d-1)/2$ and the symbol-class uniformity required in Lemma~\ref{lem:scalar-parameter-derivative}. Thus Proposition~\ref{prop:scalar-double-pinning} applies. Since
\begin{equation*}
\eta
=
s-\frac{d+1}{2},
\end{equation*}
the condition $s>(d+2)/2$ gives $\eta>\frac{1}{2}$. Applying the proposition on each patch and using the preceding localization argument shows that, for almost every pin, at least one localized pinned pushforward has a nonzero continuous density. Its support contains an interval contained in $\Delta_\Phi^y(E)$.
\end{proof}

\begin{remark}[Asymmetric configurations]
Proposition~\ref{prop:scalar-double-pinning} also applies when the point and pin variables lie in spaces of different dimensions. This includes, on regular patches, point--hyperplane, point--sphere, point--line and line--line distance configurations from \cite{GIT19}. The threshold is
\begin{equation*}
\gamma+\frac{s_X-d_X}{2}+\frac{s_Y-d_Y}{2}>0
\end{equation*}
for positive measure, and the same inequality with right-hand side $\frac{1}{2}$ for nonempty interior.
\end{remark}

\section{Examples of scalar configuration maps}\label{sec:scalar-examples}

We now verify the hypotheses of Proposition~\ref{prop:scalar-double-pinning} for the scalar configuration maps appearing in the introduction. In each example, the fixed-parameter estimate follows from the standard Fourier-integral smoothing theorem on regular patches, while Lemma~\ref{lem:scalar-parameter-derivative} supplies the estimate for differentiation in the level parameter.

\subsection{Generalized norm distances}

\begin{proof}[Proof of Corollary~\ref{cor:intro-generalized-distances}]
Let
\begin{equation*}
\Phi_B(x,y)=\|x-y\|_B.
\end{equation*}
For $t$ in a compact subinterval of $(0,\infty)$, the associated operators average over translates of $t\partial B$. The smoothness and nonvanishing Gaussian curvature of $\partial B$, together with stationary phase, give
\begin{equation*}
\|T_t^\chi f_i\|_2
\lesssim
2^{-i(d-1)/2}\|f_i\|_2
\end{equation*}
on every regular patch away from the diagonal. Lemma~\ref{lem:scalar-parameter-derivative} gives
\begin{equation*}
\|\partial_tT_t^\chi f_i\|_2
\lesssim
2^i2^{-i(d-1)/2}\|f_i\|_2.
\end{equation*}
The diagonal has $\mu\times\mu$ measure zero because a finite-energy measure has no atoms. Consequently, the regular patches cover $E\times E$ up to a $\mu\times\mu$ null set. Corollary~\ref{cor:equal-dimensional-scalar}, followed by the localization argument in Section~\ref{sec:scalar-positive}, gives the conclusion with $\gamma=(d-1)/2$.
\end{proof}

The unpinned nonempty-interior conclusion at $s>(d+1)/2$ is the generalized Mattila--Sj\"olin theorem of \cite{IosevichMourgoglouTaylor}, also recovered by \cite[Corollary~1.2]{GIT19}. The pinned positive-measure conclusion at the same threshold belongs to the framework of \cite{IosevichTaylorUriarteTuero}. The parameter-derivative estimate gives the pinned nonempty-interior conclusion at $s>(d+2)/2$.

\subsection{Variable-coefficient and Riemannian distances}

\begin{proof}[Proof of Corollary~\ref{cor:intro-variable-distances}]
On each regular patch, the assumptions on $\phi$ imply that the associated generalized Radon transforms are Fourier integral operators of order $-(d-1)/2$ associated with local canonical graphs. Consequently,
\begin{equation*}
\|T_t^\chi f_i\|_2
\lesssim
2^{-i(d-1)/2}\|f_i\|_2.
\end{equation*}
Lemma~\ref{lem:scalar-parameter-derivative} supplies
\begin{equation*}
\|\partial_tT_t^\chi f_i\|_2
\lesssim
2^i2^{-i(d-1)/2}\|f_i\|_2.
\end{equation*}
Proposition~\ref{prop:scalar-double-pinning} therefore gives the localized pinned nonempty-interior conclusion when $s>(d+2)/2$. The countable regular-patch hypothesis and the localization argument in Section~\ref{sec:scalar-positive} give the global variable-coefficient conclusion.

For the Riemannian distance function, the same argument applies on the coordinate patches in the statement, where the distance is smooth and the associated spherical mean operators have canonical graph relations. The assumed countable covering up to a $\mu\times\mu$ null set permits the same global passage and proves the Riemannian conclusion.
\end{proof}

The pinned positive-measure result at the lower threshold is due to Iosevich, Taylor and Uriarte-Tuero \cite{IosevichTaylorUriarteTuero}; the corresponding local unpinned interval theorem appears in \cite[Corollary~1.3]{GIT19}.

\subsection{Dot products and bilinear forms}

\begin{proof}[Proof of Corollary~\ref{cor:intro-dot-products}]
Let
\begin{equation*}
\Phi_A(x,y)=x^TAy.
\end{equation*}
The gradients are
\begin{equation*}
\nabla_x\Phi_A(x,y)=Ay,
\qquad
\nabla_y\Phi_A(x,y)=A^Tx.
\end{equation*}
Since $A$ is invertible, these gradients do not vanish on a nonzero level. A direct computation of the Phong--Stein rotational curvature shows that the associated generalized Radon transforms are nondegenerate on every compact patch on which $x^TAy$ is bounded away from zero. Hence
\begin{equation*}
\|T_t^\chi f_i\|_2
\lesssim
2^{-i(d-1)/2}\|f_i\|_2.
\end{equation*}
Lemma~\ref{lem:scalar-parameter-derivative} gives
\begin{equation*}
\|\partial_tT_t^\chi f_i\|_2
\lesssim
2^i2^{-i(d-1)/2}\|f_i\|_2.
\end{equation*}
The null-set hypothesis in Corollary~\ref{cor:intro-dot-products} ensures that countably many nonzero-level patches cover $E\times E$ up to a set of $\mu\times\mu$ measure zero. Proposition~\ref{prop:scalar-double-pinning} and the localization argument in Section~\ref{sec:scalar-positive} therefore give the conclusion.
\end{proof}

For $A=I$, this gives a pinned dot-product statement for $\Phi(x,y)=x\cdot y$. In that case this is already a consequence of the work of Bright, Marshall and Senger, see Theorem 1.2 (3) in \cite{BrightMarshallSenger}.

\section{Geometric reduction to triangles}\label{sec:geometric-reduction}

The induction step uses the following elementary energy argument. It is a standard energy-averaging argument in the spirit of classical projection estimates; see, for instance, \cite{Marstrand,Kaufman,MattilaBook,Mattila,OrponenRadial}.

\begin{lemma}[Radial directions with large orthogonal projections]\label{lem:projection}
Let $E\subset\R^d$ be compact and contained in an annulus centered at the origin. Let
\begin{equation*}
\Omega=\pi_0(E)=\left\{\frac{x}{|x|}:x\in E\right\}\subset S^{d-1}.
\end{equation*}
Suppose $\beta<\dimH(E)$ and $\beta<\dimH(\Omega)$. Then there exists $\theta\in\Omega$ such that
\begin{equation*}
\dimH(P_{\theta^\perp}E)\geq \beta.
\end{equation*}
\end{lemma}

\begin{proof}
Choose a probability measure $\mu$ supported on $E$ with finite $\beta$-energy:
\begin{equation*}
I_\beta(\mu)=\iint |x-y|^{-\beta}\,d\mu(x)d\mu(y)<\infty.
\end{equation*}
Choose a probability measure $\sigma$ supported on $\Omega$ and a number $\gamma>\beta$ such that
\begin{equation*}
\sigma(B(\omega,\rho))\lesssim \rho^\gamma
\end{equation*}
for every $\omega\in S^{d-1}$ and $\rho>0$.

We estimate
\begin{equation*}
\int_\Omega I_\beta((P_{\theta^\perp})_\#\mu)\,d\sigma(\theta)
=
\iint
\int_\Omega |P_{\theta^\perp}(x-y)|^{-\beta}\,d\sigma(\theta)\,d\mu(x)d\mu(y).
\end{equation*}
For $v=x-y\neq0$,
\begin{equation*}
|P_{\theta^\perp}v|=|v|\sin\angle(\theta,v).
\end{equation*}
Thus
\begin{equation*}
\int_\Omega |P_{\theta^\perp}v|^{-\beta}\,d\sigma(\theta)
=
|v|^{-\beta}
\int_\Omega \sin(\angle(\theta,v))^{-\beta}\,d\sigma(\theta).
\end{equation*}
The singularity occurs when $\theta$ is close to either $v/|v|$ or $-v/|v|$. Decomposing into dyadic angular annuli around these two points, and using the $\gamma$-Frostman condition, gives
\begin{align*}
\int_\Omega \sin(\angle(\theta,v))^{-\beta}\,d\sigma(\theta)
&\lesssim
\sum_{j\geq0}2^{j\beta}\sigma(B(v/|v|,C2^{-j}))
\\
&\quad+
\sum_{j\geq0}2^{j\beta}\sigma(B(-v/|v|,C2^{-j}))
\\
&\lesssim
\sum_{j\geq0}2^{j(\beta-\gamma)}
<\infty.
\end{align*}
Therefore
\begin{equation*}
\int_\Omega |P_{\theta^\perp}v|^{-\beta}\,d\sigma(\theta)
\lesssim |v|^{-\beta}.
\end{equation*}
It follows that
\begin{equation*}
\int_\Omega I_\beta((P_{\theta^\perp})_\#\mu)\,d\sigma(\theta)
\lesssim I_\beta(\mu)<\infty.
\end{equation*}
Hence for $\sigma$-almost every $\theta\in\Omega$,
\begin{equation*}
I_\beta((P_{\theta^\perp})_\#\mu)<\infty.
\end{equation*}
For such a $\theta$, the support of $(P_{\theta^\perp})_\#\mu$ is contained in $P_{\theta^\perp}E$, and finite $\beta$-energy implies
\begin{equation*}
\dimH(P_{\theta^\perp}E)\geq\beta.
\end{equation*}
\end{proof}

\begin{proof}[Proof of Theorem~\ref{thm:geometric-reduction}]
We prove the positive-measure and nonempty-interior statements in parallel. Fix one of the two conclusions, and write $X(m)$ for the corresponding threshold, namely $X_{\mathrm{pm}}(m)$ in the positive-measure case and $X_{\mathrm{int}}(m)$ in the nonempty-interior case. In either case, the conclusion is preserved under multiplication by a nonzero scalar and under passage to a larger set.

We first establish an auxiliary strongly pinned assertion. For $2\leq r\leq m$, suppose that $F\subset\R^m$ is compact and
\begin{equation*}
\dimH(F)>X(m-r+2)+r-2.
\end{equation*}
We claim that $\Vol_r^0(F)$ has the chosen conclusion. The case $r=2$ is the corresponding assumed triangle theorem.

Assume $r>2$ and that the claim is known for rank $r-1$ in dimension $m-1$. Choose numbers $\alpha$ and $\beta$ such that
\begin{equation*}
X(m-r+2)+r-3<\beta<\alpha<\dimH(F)-1.
\end{equation*}
By countable stability of Hausdorff dimension, $F$ has a compact subset, still denoted by $F$, contained in an annulus centered at the origin and satisfying $\dimH(F)>\alpha+1$. If
\begin{equation*}
\Omega=\left\{\frac{x}{|x|}:x\in F\right\},
\end{equation*}

On the chosen annulus, the polar-coordinate map
\begin{equation*}
x\longmapsto\left(\frac{x}{|x|},|x|\right)
\end{equation*}
is bi-Lipschitz onto its image. Since the radial coordinate is one-dimensional, the standard product-dimension inequality gives
\begin{equation*}
\dimH(\Omega)
\geq
\dimH(F)-1
>
\alpha
>
\beta.
\end{equation*}

Lemma~\ref{lem:projection} therefore supplies $\theta\in\Omega$ with
\begin{equation*}
\dimH(P_{\theta^\perp}F)\geq\beta>X(m-r+2)+r-3.
\end{equation*}
Choose $\rho>0$ with $\rho\theta\in F$. Since
\begin{equation*}
|x_1\wedge\cdots\wedge x_{r-1}\wedge\rho\theta|
=
\rho\,|P_{\theta^\perp}x_1\wedge\cdots\wedge P_{\theta^\perp}x_{r-1}|,
\end{equation*}
the induction hypothesis in $\theta^\perp\simeq\R^{m-1}$ applies, because its threshold is
\begin{equation*}
X((m-1)-(r-1)+2)+(r-1)-2
=
X(m-r+2)+r-3.
\end{equation*}
Thus $\Vol_r^0(F)$ has the chosen conclusion, proving the auxiliary assertion.

We now return to the doubly pinned problem. Translate so that $x_0=0$, fix $y\neq0$, let $H=y^\perp$, and set $F=P_HE$. 

The fibers of $P_H$ are affine lines. The standard dimension inequality for Lipschitz maps therefore gives
\begin{equation*}
\dimH(E)
\leq
\dimH(P_HE)+1.
\end{equation*}
Consequently,
\begin{equation*}
\dimH(F)
\geq
\dimH(E)-1
>
X(d-k+2)+k-3.
\end{equation*}

This is exactly the auxiliary threshold for rank $k-1$ in dimension $d-1$. Moreover,
\begin{equation*}
\Vol_k^{0,y}(E)=|y|\,\Vol_{k-1}^0(F).
\end{equation*}
Indeed, components parallel to $y$ vanish in the wedge product, and every projected configuration lifts by choosing preimages in $E$. The auxiliary assertion gives the chosen conclusion for $\Vol_{k-1}^0(F)$, and multiplication by $|y|>0$ preserves it. Applying this argument first with $X=X_{\mathrm{pm}}$ and then with $X=X_{\mathrm{int}}$ proves both parts of the theorem.
\end{proof}
\begin{lemma}[Nullity of orthogonal pairs]\label{lem:orthogonal-pairs-null}
Let $\lambda$ and $\eta$ be compactly supported Frostman measures on $\R^d$ with exponents $a>0$ and $b>0$. If
\begin{equation*}
a+b>d,
\end{equation*}
then
\begin{equation*}
(\lambda\times\eta)
\bigl(\{(x,y)\in\R^d\times\R^d:x\cdot y=0\}\bigr)
=
0.
\end{equation*}
The same conclusion holds if $I_a(\lambda)<\infty$ and $I_b(\eta)<\infty$.
\end{lemma}

\begin{proof}
Suppose that the orthogonality set has positive product measure. Since positive-exponent Frostman measures are nonatomic, Fubini's theorem gives a set $X\subset\R^d\setminus\{0\}$ of positive $\lambda$-measure such that
\begin{equation*}
\eta(x^\perp)>0
\end{equation*}
for every $x\in X$.

Call a linear subspace $V$ $\eta$-minimal if $\eta(V)>0$ but every proper linear subspace of $V$ has zero $\eta$-measure. Every positive-$\eta$ subspace contains an $\eta$-minimal subspace. Moreover, there are at most countably many $\eta$-minimal subspaces, because distinct minimal subspaces have $\eta$-null intersection and hence their indicator functions form a pairwise orthogonal family in $L^2(\eta)$. Enumerate them as $\{V_j\}_{j\geq1}$.

For every $x\in X$, the positive-$\eta$ subspace $x^\perp$ contains some $V_j$. Hence
\begin{equation*}
X\subset\bigcup_{j\geq1}V_j^\perp.
\end{equation*}
It follows that, for some $j$,
\begin{equation*}
\lambda(V_j^\perp)>0,
\qquad
\eta(V_j)>0.
\end{equation*}
The Frostman conditions imply
\begin{equation*}
a\leq\dim V_j^\perp,
\qquad
b\leq\dim V_j.
\end{equation*}
Therefore
\begin{equation*}
a+b
\leq
\dim V_j^\perp+\dim V_j
=
d,
\end{equation*}
contrary to the hypothesis.

If $I_a(\lambda)<\infty$ and $I_b(\eta)<\infty$, Lemma~\ref{lem:energy-trace} gives increasing Borel sets $E_N$ and $F_M$ such that
\begin{equation*}
\lambda\left(\R^d\setminus\bigcup_{N\geq1}E_N\right)=0,
\qquad
\eta\left(\R^d\setminus\bigcup_{M\geq1}F_M\right)=0,
\end{equation*}
and such that $\lambda|_{E_N}$ is $a$-Frostman and $\eta|_{F_M}$ is $b$-Frostman for every $N,M$. Applying the preceding argument to each pair
\begin{equation*}
\lambda|_{E_N},
\qquad
\eta|_{F_M},
\end{equation*}
shows that the orthogonality set has zero $(\lambda|_{E_N})\times(\eta|_{F_M})$-measure. Since
\begin{equation*}
\bigcup_{N,M\geq1}E_N\times F_M
\end{equation*}
has full $\lambda\times\eta$ measure, the orthogonality set has zero $\lambda\times\eta$ measure.

\end{proof}

\section{The triangle-area estimate}\label{sec:triangle-area}

We prove Theorem \ref{thm:area}. By translating, we assume $x_0=0$. Set
\begin{equation*}
A(x_1,x_2)=|x_1\wedge x_2|.
\end{equation*}
We shall use the localized Leray-measure operators associated to the level relation
\begin{equation*}
A(x_1,x_2)=t.
\end{equation*}
The localization is part of the definition of the operator, and it will be chosen on a regular patch where the geometry of the fibers is uniform.

Fix compact intervals $I,I^*\subset(0,\infty)$ with
$I\subset\operatorname{int}(I^*)$, and fix
$\chi\in C_c^\infty(\R^d\times\R^d)$ supported where
\begin{equation*}
0<a\leq |x_2|\leq b<\infty,
\qquad
|x_1\wedge x_2|\in I^*,
\qquad
|x_1\cdot x_2|\geq c_0>0.
\end{equation*}
For $x_2\neq0$, write $x_2=r\theta$, where $r=|x_2|$ and $\theta=x_2/|x_2|\in S^{d-1}$. We denote by $P_{\theta^\perp}$ the orthogonal projection onto the hyperplane $\theta^\perp$. Then
\begin{equation*}
A(x_1,x_2)=r|P_{\theta^\perp}x_1|.
\end{equation*}
Thus, on the support of $\chi$, the level set
\begin{equation*}
\Sigma_{x_2,t}=\{x_1\in\R^d:A(x_1,x_2)=t\}
\end{equation*}
is a smooth hypersurface for every $t\in I$. Indeed, on this region $A(x_1,x_2)>0$, hence $P_{\theta^\perp}x_1\neq0$. Since $P_{\theta^\perp}$ is self-adjoint, for every $h\in\R^d$,
\begin{equation*}
d_{x_1}A(x_1,x_2)[h]
=
r\frac{P_{\theta^\perp}x_1}{|P_{\theta^\perp}x_1|}\cdot P_{\theta^\perp}h
=
r\frac{P_{\theta^\perp}x_1}{|P_{\theta^\perp}x_1|}\cdot h.
\end{equation*}
Therefore
\begin{equation*}
\nabla_{x_1}A(x_1,x_2)=r\frac{P_{\theta^\perp}x_1}{|P_{\theta^\perp}x_1|},
\qquad
|\nabla_{x_1}A(x_1,x_2)|=r=|x_2|.
\end{equation*}
The lower bound $|x_2|\geq a$ on the support of $\chi$ gives the required regularity of the level set.

For $f\in C_c^\infty(\R^d)$, define
\begin{equation*}
T_t^\chi f(x_2)
=
\int_{\Sigma_{x_2,t}}
f(x_1)\chi(x_1,x_2)
\frac{d\mathcal H^{d-1}(x_1)}{|\nabla_{x_1}A(x_1,x_2)|}.
\end{equation*}
Equivalently,
\begin{equation*}
T_t^\chi f(x_2)
=
\int_{\R^d}
f(x_1)\chi(x_1,x_2)\delta(A(x_1,x_2)-t)\,dx_1,
\end{equation*}
where the delta notation is understood in the Leray-measure sense. In the proof below we suppress the cutoff from the notation and write $T_t$ when no confusion can arise.

We use the level-parameter regularization introduced in Section~\ref{sec:radon-preliminaries}. In the present setting, the coarea formula gives
\begin{equation*}
T_{t,\eps}^\chi f(x_2)
=
\int_{\R^d}
f(x_1)\chi(x_1,x_2)\rho_\eps(A(x_1,x_2)-t)\,dx_1.
\end{equation*}
Thus the smoothing is applied to the localized level-set measure in the parameter $t$, while the Frostman measure remains fixed.

\begin{lemma}[Cylinder averaging estimate]\label{lem:cylinder}
Let $d\geq2$, let $I\subset(0,\infty)$ be a compact interval, and let $T_t^\chi$ be the localized triangle-area Radon transform defined above. Suppose
\begin{equation*}
\supp\widehat f_i\subset\{\xi\in\R^d:|\xi|\approx2^i\}.
\end{equation*}
Then, uniformly for $t\in I$,
\begin{equation*}
\|T_t^\chi f_i\|_2
\lesssim_{\chi,I}
2^{-i(d-1)/2}\|f_i\|_2.
\end{equation*}
Moreover, uniformly for $t,t'\in I$,
\begin{equation*}
\|(T_t^\chi-T_{t'}^\chi)f_i\|_2
\lesssim_{\chi,I}
\min(1,2^i|t-t'|)2^{-i(d-1)/2}\|f_i\|_2.
\end{equation*}
The same estimates hold uniformly for the regularized operators $T_{t,\eps}^\chi$, for $0<\eps<\eps_0$, where $\eps_0$ is chosen so that the support of $\rho_\eps(\tau-t)$ remains inside $I^*$ whenever $t\in I$.
\end{lemma}

\begin{proof}
Set
\begin{equation*}
\Psi(x_1,x_2)
=
|x_1\wedge x_2|^2
=
|x_1|^2|x_2|^2-(x_1\cdot x_2)^2.
\end{equation*}
Since $t$ is bounded away from zero, the relation
\begin{equation*}
|x_1\wedge x_2|=t
\end{equation*}
is equivalent to
\begin{equation*}
\Psi(x_1,x_2)=t^2,
\end{equation*}
and
\begin{equation*}
\delta(|x_1\wedge x_2|-t)
=
2t\,\delta(\Psi(x_1,x_2)-t^2).
\end{equation*}
The factor $2t$ is harmless on the fixed compact interval $I$. Moreover, on the region where $A(x_1,x_2)>0$,
\begin{equation*}
d\Psi=2A\,dA.
\end{equation*}
Hence the level relations $\Psi=t^2$ and $A=t$ have the same conormal bundle,
up to multiplication of the conormal parameter by the nonzero factor $2t$.
Thus they determine the same canonical relation.

We verify the rotational-curvature condition for $\Psi$. Direct differentiation gives
\begin{equation*}
\nabla_{x_1}\Psi
=
2\bigl(|x_2|^2x_1-(x_1\cdot x_2)x_2\bigr),
\end{equation*}
\begin{equation*}
\nabla_{x_2}\Psi
=
2\bigl(|x_1|^2x_2-(x_1\cdot x_2)x_1\bigr),
\end{equation*}
and
\begin{equation*}
D_{x_2}D_{x_1}\Psi
=
4x_1x_2^T
-
2x_2x_1^T
-
2(x_1\cdot x_2)I_d.
\end{equation*}
Consequently,
\begin{equation}\label{eq:area-rotational-curvature}
\det
\begin{pmatrix}
0 & \nabla_{x_1}\Psi^T \\
\nabla_{x_2}\Psi & D_{x_2}D_{x_1}\Psi
\end{pmatrix}
=
(-1)^d2^{d+2}\Psi^2(x_1\cdot x_2)^{d-2}.
\end{equation}

To verify \eqref{eq:area-rotational-curvature}, apply a simultaneous
orthogonal transformation and write
\begin{equation*}
x_2=re_d,
\qquad
x_1=ve_1+ue_d.
\end{equation*}
Then
\begin{equation*}
\Psi=r^2v^2,
\qquad
x_1\cdot x_2=ur.
\end{equation*}
In the orthogonal complement of $\operatorname{span}\{e_1,e_d\}$, the mixed
Hessian is $-2urI_{d-2}$. The remaining bordered block is
\begin{equation*}
\begin{pmatrix}
0&2r^2v&0\\
-2urv&-2ur&4rv\\
2rv^2&-2rv&0
\end{pmatrix},
\end{equation*}
whose determinant is $16r^4v^4$. Hence the full determinant is
\begin{equation*}
16r^4v^4(-2ur)^{d-2}
=
(-1)^d2^{d+2}\Psi^2(x_1\cdot x_2)^{d-2}.
\end{equation*}

On the support of $\chi$, both $\Psi$ and $|x_1\cdot x_2|$ are bounded away from zero. Hence the determinant in
\eqref{eq:area-rotational-curvature} is uniformly nonzero. The nonvanishing rotational-curvature determinant implies that the associated
canonical relation is locally a canonical graph. Since a codimension-one
generalized Radon transform associated with a local canonical graph gains
$\frac{d-1}{2}$ derivatives on $L^2$ Sobolev spaces
\cite[Section~1, equations~(1.5)--(1.7)]{GIT19}, we obtain
\begin{equation*}
\|T_t^\chi f_i\|_2
\lesssim_{\chi,I}
2^{-i(d-1)/2}\|f_i\|_2
\end{equation*}
uniformly for $t\in I^*$.

The symbol seminorms and the rotational-curvature lower bound are uniform for $t$ in the compact interval $I^*$. Lemma~\ref{lem:scalar-parameter-derivative} therefore gives
\begin{equation*}
\|\partial_tT_t^\chi f_i\|_2
\lesssim_{\chi,I}
2^i2^{-i(d-1)/2}\|f_i\|_2.
\end{equation*}
Combining this estimate with the fixed-parameter bound and the fundamental theorem of calculus yields
\begin{equation*}
\|(T_t^\chi-T_{t'}^\chi)f_i\|_2
\lesssim_{\chi,I}
\min\{1,2^i|t-t'|\}
2^{-i(d-1)/2}\|f_i\|_2.
\end{equation*}
The corresponding estimates for $T_{t,\eps}^\chi$ follow from Lemma~\ref{lem:scalar-parameter-derivative}.
\end{proof}

\begin{lemma}[Frequency compatibility]\label{lem:output-frequency}
Let $P_i$ and $P_j$ be smooth Littlewood--Paley projections. For the localized triangle-area operators above, there is $C_0\geq1$ such that, for every $M>0$ and $j\geq i+C_0$,
\begin{equation*}
\|P_jT_{t,\eps}^\chi P_i f\|_2
\lesssim_{M,\chi,I}
2^{-M(j-i)}2^{-i(d-1)/2}\|P_i f\|_2,
\end{equation*}
uniformly for $t\in I$ and $0<\eps<\eps_0$. The analogous estimate holds for $T_{t,\eps}^\chi-T_{t',\eps}^\chi$, with the additional factor $\min(1,2^i|t-t'|)$.
\end{lemma}

\begin{proof}
For the unregularized family, the localized kernel is a Fourier integral distribution whose canonical relation is contained in a fixed conic set on which input and output covectors have comparable size. If $j\geq i+C_0$, the phase has no stationary point after the input and output projections are inserted. Repeated integration by parts gives arbitrary decay in $2^{j-i}$. On the complementary frequency region, Lemma~\ref{lem:cylinder} supplies the factor $2^{-i(d-1)/2}$. Differentiating or taking a difference in $t$ changes only the amplitude order and gives the factor $\min(1,2^i|t-t'|)$ established in Lemma~\ref{lem:cylinder}. Finally,
\begin{equation*}
T_{t,\eps}^\chi=\int \rho_\eps(\tau-t)T_\tau^\chi\,d\tau,
\end{equation*}
so the same estimates pass uniformly to the regularized family by Minkowski's inequality.
\end{proof}

\begin{proof}[Proof of Theorem \ref{thm:area}]
Choose $s$ with
\begin{equation*}
\frac{d+1}{2}<s<\dimH(E),
\end{equation*}
and choose a probability measure $\mu$ supported on $E$ with finite $s$-energy,
\begin{equation*}
I_s(\mu)=\int\int |x-y|^{-s}\,d\mu(x)d\mu(y)<\infty.
\end{equation*}
Let $P_i$ be a fixed smooth Littlewood--Paley projection to the annulus $|\xi|\approx2^i$, and set $\mu_i=P_i\mu$. The finite energy assumption gives
\begin{equation*}
\|\mu_i\|_2\lesssim2^{i(d-s)/2}.
\end{equation*}

We next choose a positive-mass patch where the level-set geometry is uniform. Since $s>1$, the measure $\mu$ gives zero mass to every line through the origin, and hence
\begin{equation*}
(\mu\times\mu)
\bigl(\{(x_1,x_2):x_1\wedge x_2=0\}\bigr)
=
0.
\end{equation*}
Moreover, $2s>d$, so Lemma~\ref{lem:orthogonal-pairs-null} gives
\begin{equation*}
(\mu\times\mu)
\bigl(\{(x_1,x_2):x_1\cdot x_2=0\}\bigr)
=
0.
\end{equation*}

Therefore there are compact intervals $I,I^*\subset(0,\infty)$ satisfying
\begin{equation*}
I\subset\operatorname{int}(I^*),
\end{equation*}
constants
\begin{equation*}
0<a<b<\infty,
\qquad
c_0>0,
\end{equation*}
and a nonnegative cutoff $\chi\in C_c^\infty(\R^d\times\R^d)$ supported where
\begin{equation*}
a\leq |x_2|\leq b,
\qquad
A(x_1,x_2)\in\operatorname{int}(I),
\qquad
|x_1\cdot x_2|\geq c_0,
\end{equation*}
such that
\begin{equation*}
\int\chi(x_1,x_2)\,d\mu(x_1)d\mu(x_2)>0.
\end{equation*}

Define the localized pinned area measure by
\begin{equation*}
\int g(t)\,d\nu_\chi(t)
=
\int g(A(x_1,x_2))\chi(x_1,x_2)\,d\mu(x_1)d\mu(x_2).
\end{equation*}
Let $\rho_\eps$ be the approximate identity chosen above and set
\begin{equation*}
\nu_{\chi,\eps}(t)
=
\int \rho_\eps(A(x_1,x_2)-t)\chi(x_1,x_2)\,d\mu(x_1)d\mu(x_2).
\end{equation*}

For $t\in I$ and $0<\eps<\eps_0$, the kernel
\begin{equation*}
K_{t,\eps}(x_1,x_2)
=
\chi(x_1,x_2)\rho_\eps(A(x_1,x_2)-t)
\end{equation*}
is smooth and compactly supported on the chosen regular patch. Thus the notation
\begin{equation*}
\nu_{\chi,\eps}(t)
=
\langle T_{t,\eps}^\chi\mu,\mu\rangle
\end{equation*}
means precisely the preceding double integral, or equivalently the pairing of
$\mu\otimes\mu$ with the smooth kernel $K_{t,\eps}$.

For the estimates, we use this identity after Littlewood--Paley decomposition. With $\mu_i=P_i\mu$, we first write the identity for finite partial sums and then pass to the limit using the estimates below:
\begin{equation*}
\nu_{\chi,\eps}(t)
=
\sum_{i,j\geq0}
\langle T_{t,\eps}^\chi\mu_i,\mu_j\rangle.
\end{equation*}

By Cauchy--Schwarz, Lemma \ref{lem:cylinder}, and the Frostman bound,
\begin{equation*}
|\langle T_{t,\eps}^\chi\mu_i,\mu_j\rangle|
\lesssim_{\chi,I}
2^{-i(d-1)/2}2^{i(d-s)/2}2^{j(d-s)/2}.
\end{equation*}
By Lemma~\ref{lem:output-frequency}, the localized operator has the following output-frequency compatibility: for every $M>0$,
\begin{equation*}
|\langle T_{t,\eps}^\chi\mu_i,\mu_j\rangle|
\lesssim_{\chi,I,M}
2^{-i(d-1)/2}2^{-M(j-i)_+}\|\mu_i\|_2\|\mu_j\|_2,
\end{equation*}
where $u_+=\max(u,0)$. Thus the contribution of $j>i+C_0$ is rapidly decaying. Setting $i=N$ and summing over $j\leq N+C_0$ gives
\begin{equation*}
\sum_{j\leq N+C_0}
|\langle T_{t,\eps}^\chi\mu_N,\mu_j\rangle|
\lesssim_{\chi,I}
2^{-N(d-1)/2}2^{N(d-s)/2}
\sum_{j\leq N+C_0}2^{j(d-s)/2}.
\end{equation*}
Since $s<d$, the last sum is dominated by its largest term. Hence the contribution at frequency $N$ is
\begin{equation*}
\lesssim_{\chi,I}
2^{-N(d-1)/2}2^{N(d-s)/2}2^{N(d-s)/2}
=
2^{N(d+1-2s)/2}.
\end{equation*}
This is summable because $s>(d+1)/2$. Therefore the regularized densities $\nu_{\chi,\eps}$ are uniformly bounded on $I$, independently of $\eps$.

We now prove parameter regularity for the regularized densities. Choose
\begin{equation*}
0<\gamma<\min\left\{1,\ s-\frac{d+1}{2}\right\}.
\end{equation*}
Using the difference estimate in Lemma~\ref{lem:cylinder}, Cauchy--Schwarz, the Frostman bound, and the same output-frequency compatibility as above, the contribution with top input frequency $N$ satisfies
\begin{equation*}
\sum_{j\leq N+C_0}
\left|
\left\langle
(T_{t,\eps}^\chi-T_{t',\eps}^\chi)\mu_N,\mu_j
\right\rangle
\right|
\lesssim_{\chi,I}
\min(1,2^N|t-t'|)
2^{-N(d-1)/2}2^{N(d-s)/2}
\sum_{j\leq N+C_0}2^{j(d-s)/2}.
\end{equation*}
Since $s<d$, this is bounded by
\begin{equation*}
\lesssim_{\chi,I}
\min(1,2^N|t-t'|)2^{N(d+1-2s)/2}.
\end{equation*}
As $0<\gamma\leq1$,
\begin{equation*}
\min(1,2^N|t-t'|)
\leq
(2^N|t-t'|)^\gamma.
\end{equation*}
Therefore the contribution with top input frequency $N$ is
\begin{equation*}
\lesssim_{\chi,I,\gamma}
|t-t'|^\gamma
2^{N((d+1-2s)/2+\gamma)}.
\end{equation*}
The exponent is negative by the choice of $\gamma$, so summing in $N$ gives
\begin{equation*}
|\nu_{\chi,\eps}(t)-\nu_{\chi,\eps}(t')|
\lesssim_{\chi,I,\gamma}
|t-t'|^\gamma,
\end{equation*}
uniformly in $\eps$ and uniformly for $t,t'\in I$. Thus the family $\{\nu_{\chi,\eps}\}_{\eps>0}$ is uniformly bounded and uniformly H\"older continuous on $I$.

By Arzel\`a--Ascoli, a sequence $\eps_m\to0$ may be chosen so that $\nu_{\chi,\eps_m}$ converges uniformly on $I$ to a continuous nonnegative function $h$. On the other hand, by the definition of the regularization, the measures $\nu_{\chi,\eps}(t)\,dt$ converge weakly to $\nu_\chi$. Hence, on $I$,
\begin{equation*}
d\nu_\chi(t)=h(t)\,dt.
\end{equation*}
Moreover,
\begin{equation*}
\nu_\chi(\R)
=
\int \chi(x_1,x_2)\,d\mu(x_1)d\mu(x_2)>0.
\end{equation*}

Since $\Phi(\supp\chi)\subset\operatorname{int}(I)$, we have
\begin{equation*}
\nu_\chi(I)=\nu_\chi(\R)>0.
\end{equation*}

Therefore $h$ is not identically zero. Since $h$ is continuous and nonnegative, there are $t_0\in I$ and an open interval $J\subset I$ with $h(t)>0$ for every $t\in J$. It follows that every nonempty open subinterval of $J$ has positive $\nu_\chi$-measure, and hence
\begin{equation*}
J\subset \supp \nu_\chi.
\end{equation*}
The support of $\nu_\chi$ is contained in the closure of the localized area values
\begin{equation*}
\{A(x_1,x_2):(x_1,x_2)\in (E\times E)\cap \supp\chi\}.
\end{equation*}
Since $A$ is continuous and $(E\times E)\cap\supp\chi$ is compact, this set of localized area values is compact and therefore closed. Thus $J\subset \Vol_2^0(E)$. This proves that $\Vol_2^0(E)$ has nonempty interior.
\end{proof}

\begin{remark}
When $d=2$, the cross-section $S^{d-2}$ is $S^0$, and the estimate reduces to the line-averaging bound
\begin{equation*}
\|T_t f_i\|_2\lesssim 2^{-i/2}\|f_i\|_2.
\end{equation*}
For $d>2$, the additional factor $2^{-i(d-2)/2}$ comes from the Fourier decay of the spherical cross-section $S^{d-2}$.
\end{remark}

\subsection{Doubly pinned triangle areas}

\begin{proof}[Proof of Theorem~\ref{thm:double-pinned-area}]
After translating, assume $x_0=0$. Choose a countable family of nonnegative cutoffs $\chi_n$ supported on compact regular patches of the form
\begin{equation*}
0<a\leq |y|\leq b,
\qquad
0<c\leq |x\wedge y|\leq C,
\qquad
|x\cdot y|\geq c_0>0,
\end{equation*}
such that
\begin{equation*}
\sum_{n\geq1}\chi_n(x,y)>0
\end{equation*}
whenever
\begin{equation*}
y\neq0,
\qquad
x\wedge y\neq0,
\qquad
x\cdot y\neq0.
\end{equation*}

On each patch, Lemma~\ref{lem:cylinder} gives the fixed-parameter estimate with smoothing order
\begin{equation*}
\gamma=\frac{d-1}{2},
\end{equation*}
together with the required parameter-difference estimates, while Lemma~\ref{lem:output-frequency} gives the output-frequency compatibility. Hence Proposition~\ref{prop:scalar-double-pinning} applies with $X=Y=\R^d$ and with both the point and pin measures equal to $\mu$. In this case
\begin{equation*}
\eta
=
\frac{d-1}{2}+s-d
=
s-\frac{d+1}{2}.
\end{equation*}
Thus, when $s>(d+1)/2$, every nonzero localized pinned pushforward has an $L^2$ density for $\mu$-almost every $y$, and when $s>(d+2)/2$, it has a continuous density.

Since $s>\frac{d+1}{2}>1$, the measure $\mu$ gives zero mass to every line. Therefore, for every $y\neq0$,
\begin{equation*}
\mu\bigl(\{x:x\wedge y=0\}\bigr)=0.
\end{equation*}
Moreover, $2s>d$, so Lemma~\ref{lem:orthogonal-pairs-null} and Fubini's theorem give
\begin{equation*}
\mu\bigl(\{x:x\cdot y=0\}\bigr)=0
\end{equation*}
for $\mu$-almost every $y$. After intersecting the countably many full-measure sets of pins supplied by the localized applications of the proposition, we conclude that, for $\mu$-almost every $y\neq0$, at least one localized pinned pushforward is nonzero. Its support is contained in
\begin{equation*}
\big\{|x\wedge y|:x\in E\big\}.
\end{equation*}
The positive-measure and nonempty-interior conclusions follow. Translating back proves the theorem.

\end{proof}

\section{Simplex-volume consequences}\label{sec:simplex-consequences}

\begin{proof}[Proof of Theorem~\ref{thm:double-pinned-volumes}]
Apply the nonempty-interior part of Theorem~\ref{thm:geometric-reduction} with
\begin{equation*}
X_{\mathrm{int}}(m)=\frac{m+1}{2},
\end{equation*}
as supplied by Theorem~\ref{thm:area}. The propagated threshold is
\begin{equation*}
X_{\mathrm{int}}(d-k+2)+k-2
=
\frac{d-k+3}{2}+k-2
=
\frac{d+k-1}{2}.
\end{equation*}
\end{proof}

\begin{remark}[The tetrahedral threshold]
For $k=3$, translate so that $x_0=0$, fix $y\neq0$, and set $F=P_{y^\perp}E$. Then
\begin{equation*}
\Vol_3^{0,y}(E)=|y|\,\Vol_2^0(F).
\end{equation*}
The projected triangle has one prescribed vertex and two free vertices, so the input is Theorem~\ref{thm:area}, not the doubly pinned triangle theorem. Since $F\subset y^\perp\simeq\R^{d-1}$, Theorem~\ref{thm:area} requires $\dimH(F)>d/2$. The bound $\dimH(F)\geq\dimH(E)-1$ then gives $\dimH(E)>(d+2)/2$. Prescribing a third tetrahedral vertex would instead produce a doubly pinned triangle and the interval threshold $(d+3)/2$.
\end{remark}

\begin{proof}[Proof of Corollary~\ref{cor:strongly-pinned-volumes}]
For $k=2$, this is Theorem~\ref{thm:area}. For $3\leq k<d$, choose any $y\in E\setminus\{x_0\}$ and use
\begin{equation*}
\Vol_k^{x_0,y}(E)\subset\Vol_k^{x_0}(E).
\end{equation*}
Theorem~\ref{thm:double-pinned-volumes} then gives the conclusion.
\end{proof}

\begin{remark}[What the reduction preserves]
The first additional vertex $y$ is arbitrary, and the conclusion holds for every $y\neq x_0$. The later vectors chosen in the induction are not arbitrary: they are selected through radial projection to retain sufficiently large projected dimension. This is the precise reason that the method reaches double pinning but does not directly give comparable results for a fully prescribed higher-dimensional pin frame.
\end{remark}

\begin{remark}[Propagation of the two triangle conclusions]
The positive-measure and nonempty-interior parts of Theorem~\ref{thm:geometric-reduction} propagate their respective triangle thresholds separately. A strongly pinned triangle positive-measure threshold $X_{\mathrm{pm}}(m)$ gives the higher-volume threshold
\begin{equation*}
X_{\mathrm{pm}}(d-k+2)+k-2,
\end{equation*}
while a strongly pinned triangle nonempty-interior threshold $X_{\mathrm{int}}(m)$ gives
\begin{equation*}
X_{\mathrm{int}}(d-k+2)+k-2.
\end{equation*}
In the present paper, Theorem~\ref{thm:area} gives $X_{\mathrm{int}}(m)=(m+1)/2$, and hence nonempty interior for the higher-rank volume set at $(d+k-1)/2$.
\end{remark}

\section{Future directions}\label{sec:future}

The first problem is to close the half-derivative gap in the doubly pinned scalar theorem. For triangle areas and for general scalar $\Phi$-configurations, the present argument gives positive measure at the fixed-parameter threshold and nonempty interior only after an additional half derivative in the level parameter. A genuine averaged local-smoothing estimate would lower the interval threshold from $(d+2)/2$ toward $(d+1)/2$.

A second direction is to prescribe more vertices of a simplex. The projection argument retains one arbitrary vector after the base point but chooses all subsequent directions. It is natural to seek analytic estimates that are uniform over a larger pin frame and thereby quantify triples or higher tuples of pins.

A third problem is to analyze intermediate wedge-product transforms directly. The current proof reduces
\begin{equation*}
|v_1\wedge\cdots\wedge v_k|=t
\end{equation*}
to the two-vector cylinder transform. Direct estimates for intermediate ranks could yield mixed-dimensional results, improved exceptional-set bounds, or better thresholds in ranges where projection discards useful multilinear structure.

\begin{question}[Regularity of pinned volume measures]\label{q:volume-pushforward-regularity}
Let $2<k<d$, let $\mu$ be a finite-energy measure on $E$, and fix $x_0\in\R^d$. Under the dimensional hypothesis of Corollary~\ref{cor:strongly-pinned-volumes}, does the pushforward
\begin{equation*}
\left((x_1,\ldots,x_k)\mapsto |(x_1-x_0)\wedge\cdots\wedge(x_k-x_0)|\right)_\#(\mu^{\otimes k})
\end{equation*}
have an absolutely continuous, or even continuous, density on a nontrivial interval? The projection argument proves a set-theoretic non-empty interior conclusion but does not retain regularity of the full multilinear pushforward.
\end{question}

Restricted and singular distance problems provide another natural testing ground. The diagonal restricted-distance configurations of \cite{GGPP} and the singular three-point variant of \cite{BorgesIosevichOu} share the passage from smoothing estimates to size or regularity of a scalar value set, but their multipoint diagonal geometry is not covered by the present two-point theorem. A corresponding restricted-input or multilinear pinning principle could lead to pinned positive-measure and nonempty-interior extensions of these results.

Finally, the scalar estimates may be combined with the hypergraph and pruning methods of \cite{BorgesFosterOuPalssonRomeroAcosta}. This suggests non-empty interior or positive-measure results for vectors of simplex volumes, provided the constituent simplices can be ordered so that each new scalar coordinate is controlled by an available pinned estimate.


\begin{thebibliography}{99}

\bibitem{BorgesFosterOuPalsson}
T. Borges, B. Foster, Y. Ou and E. A. Palsson,
\textit{Nonempty interior of pinned distance and tree sets},
Adv. Math. 493 (2026), 110917.

\bibitem{BorgesFosterOuPalssonRomeroAcostaDistGraph}
T. Borges, B. Foster, Y. Ou, E. A. Palsson and F. Romero Acosta,
\textit{Falconer-type results for any finite graph with multiple pins},
arXiv:2603.01954.

\bibitem{BorgesFosterOuPalssonRomeroAcosta}
T. Borges, B. Foster, Y. Ou, E. A. Palsson and F. Romero Acosta,
\textit{On volume vectors determined by hypergraphs in thin subsets of Euclidean space},
arXiv:2607.00153.

\bibitem{BorgesIosevichOu}
T. Borges, A. Iosevich and Y. Ou,
\textit{A singular variant of the Falconer distance problem},
arXiv:2306.05247.

\bibitem{BorgesOuPasquariello}
T. Borges, Y. Ou and M. Pasquariello,
\textit{From weighted paraboloid restriction to $k$-stars and distance graphs},
arXiv:2607.10574.

\bibitem{BrightMarshallSenger}
P. Bright, C. Marshall and S. Senger,
\textit{Pinned dot product set estimates},
Res. Math. Sci. 13 (2026), Paper No. 9.

\bibitem{Bourgain}
J. Bourgain, \textit{Hausdorff dimension and distance sets}, Israel J. Math. 87 (1994), 193--201.

\bibitem{DiNezzaPalatucciValdinoci}
E. Di Nezza, G. Palatucci and E. Valdinoci,
Hitchhiker's guide to the fractional Sobolev spaces,
\emph{Bull. Sci. Math.} \textbf{136} (2012), no. 5, 521--573.

\bibitem{DuGuthOuWangWilsonZhang}
X. Du, L. Guth, Y. Ou, H. Wang, B. Wilson and R. Zhang,
\textit{Weighted restriction estimates and application to Falconer distance set problem},
Amer. J. Math. 143 (2021), 175--211.

\bibitem{DuIosevichOuWangZhang}
X. Du, A. Iosevich, Y. Ou, H. Wang and R. Zhang,
\textit{An improved result for Falconer's distance set problem in even dimensions},
Math. Ann. 380 (2021), 1215--1231.

\bibitem{DuOuRenZhang}
X. Du, Y. Ou, K. Ren and R. Zhang,
\textit{New improvement to Falconer distance set problem in higher dimensions},
arXiv:2309.04103.

\bibitem{DuZhang}
X. Du and R. Zhang,
\textit{Sharp $L^2$ estimates of the Schr\"{o}dinger maximal function in higher dimensions},
\emph{Ann. of Math. (2)} \textbf{189} (2019), no. 3, 837--861.

\bibitem{ErdoganDistance}
M. B. Erdo\u{g}an, \textit{A bilinear Fourier extension theorem and applications to the distance set problem}, Int. Math. Res. Not., no. 23, (2005), 1411--1425.

\bibitem{ErdoganHartIosevich}
M. B. Erdo{\u g}an, D. Hart and A. Iosevich,
\textit{Multiparameter projection theorems with applications to sums-products and finite point configurations in the Euclidean setting},
in \textit{Recent advances in harmonic analysis and applications}, Springer Proc. Math. Stat. 25, Springer, New York, 2013, 93--103.

\bibitem{EswarathasanIosevichTaylor}
S. Eswarathasan, A. Iosevich and K. Taylor, \textit{Fourier integral operators, fractal sets, and the regular value theorem}, \emph{Adv. Math.} \textbf{228} (2011), no. 4, 2385--2402.

\bibitem{Falconer}
K. J. Falconer, \textit{On the Hausdorff dimensions of distance sets}, Mathematika 32 (1985), 206--212.

\bibitem{GaitanPalsson2026}
J. Gaitan Montejo and E. A. Palsson,
\textit{On volumes of simplices in intermediate dimensions},
arXiv:2605.22450.

\bibitem{GGPP}
J. Gaitan, A. Greenleaf, E. A. Palsson and G. Psaromiligkos,
\textit{On restricted Falconer distance sets},
Canad. J. Math. 77 (2025), no. 2, 665--682.

\bibitem{GaloMcDonald}
B. Galo and A. McDonald,
\textit{Volumes spanned by $k$-point configurations in $\R^d$},
J. Geom. Anal. 32 (2022), no. 1, Paper No. 23.

\bibitem{GIT19}
A. Greenleaf, A. Iosevich and K. Taylor,
\textit{Configuration sets with nonempty interior},
J. Geom. Anal. 31 (2021), 6662--6680.

\bibitem{GIT22}
A. Greenleaf, A. Iosevich and K. Taylor,
\textit{On $k$-point configuration sets with nonempty interior},
Mathematika 68 (2022), 163--190.

\bibitem{GIT24}
A. Greenleaf, A. Iosevich and K. Taylor,
\textit{Nonempty interior of configuration sets via microlocal partition optimization},
Math. Z. 306 (2024), Paper No. 66.

\bibitem{GIT25}
A. Greenleaf, A. Iosevich and K. Taylor,
\textit{Realizing trees of configurations in thin sets},
Pacific J. Math. 335.2 (2025), pp. 355-–372.

\bibitem{GreenleafIosevichMourgoglou}
A. Greenleaf, A. Iosevich and M. Mourgoglou,
\textit{On volumes determined by subsets of Euclidean space},
Forum Math. 27 (2015), no. 1, 635--646.

\bibitem{GrafakosGreenleafIosevichPalsson}
L. Grafakos, A. Greenleaf, A. Iosevich and E. A. Palsson,
\textit{Multilinear generalized Radon transforms and point configurations},
Forum Math. 27 (2015), 2323--2360.

\bibitem{GuthIosevichOuWang}
L. Guth, A. Iosevich, Y. Ou and H. Wang,
\textit{On Falconer's distance set problem in the plane},
Invent. Math. 219 (2020), 779--830.

\bibitem{HytonenVanNeervenVeraarWeis}
T. Hyt\"onen, J. van Neerven, M. Veraar and L. Weis,
\emph{Analysis in Banach Spaces. Volume I: Martingales and Littlewood--Paley Theory},
Ergebnisse der Mathematik und ihrer Grenzgebiete, Vol. 63,
Springer, Cham, 2016.

\bibitem{IosevichMourgoglouTaylor}
A. Iosevich, M. Mourgoglou and K. Taylor,
\textit{On the Mattila--Sj\"olin theorem for distance sets},
Ann. Acad. Sci. Fenn. Math. 37 (2012), 557--562.

\bibitem{IosevichTaylorUriarteTuero}
A. Iosevich, K. Taylor and I. Uriarte-Tuero,
\textit{Pinned geometric configurations in Euclidean space and Riemannian manifolds},
Mathematics 9 (2021), no. 15, 1802.

\bibitem{Kaufman}
R. Kaufman, \textit{On Hausdorff dimension of projections}, Mathematika 15 (1968), 153--155.

\bibitem{LiuPinned}
B. Liu,
\textit{An $L^2$-identity and pinned distance problem},
Geom. Funct. Anal. 29 (2019), no. 1, 283--294.

\bibitem{Marstrand}
J. M. Marstrand,
\textit{Some fundamental geometrical properties of plane sets of fractional dimensions},
Proc. London Math. Soc. (3) 4 (1954), 257--302.

\bibitem{MattilaBook}
P. Mattila,
\textit{Geometry of Sets and Measures in Euclidean Spaces},
Cambridge Studies in Advanced Mathematics 44, Cambridge University Press, 1995.

\bibitem{Mattila}
P. Mattila, \textit{Hausdorff dimension, orthogonal projections and intersections with planes}, Ann. Acad. Sci. Fenn. Ser. A I Math. 1 (1975), 227--244.

\bibitem{MattilaSjolin}
P. Mattila and P. Sj\"olin, \textit{Regularity of distance measures and sets}, Math. Nachr. 204 (1999), 157--162.

\bibitem{McDonaldAreas}
A. McDonald,
\textit{Areas spanned by point configurations in the plane},
Proc. Amer. Math. Soc. 149 (2021), no. 5, 2035--2049.

\bibitem{OrponenRadial}
T. Orponen, \textit{On the dimension and smoothness of radial projections}, Anal. PDE 12 (2019), 1273--1294.

\bibitem{PeresSchlag}
Y. Peres and W. Schlag, \textit{Smoothness of projections, Bernoulli convolutions, and the dimension of exceptions}, Duke Math. J. 102 (2000), 193--251.

\bibitem{ShmerkinYavicoli}
P. Shmerkin and A. Yavicoli,
\textit{On the volumes of simplices determined by a subset of $\R^d$},
Ann. Fenn. Math. 50 (2025), 97--108.

\bibitem{Wolff}
T. Wolff, \textit{Decay of circular means of Fourier transforms of measures}, Int. Math. Res. Not. 1999, no. 10, 547--567.

\end{thebibliography}
\end{document}